\documentclass[reqno,12pt]{amsart} 

\usepackage{amssymb}

\usepackage{esint}
\usepackage{bookmark}
\usepackage{tikz}
\usetikzlibrary{matrix}
\usepackage{float}
\usepackage[font=footnotesize]{caption}
\usepackage[left=1in, right=1in, bottom=1.5in]{geometry}
\newtheorem{theorem}{Theorem}[section]
\newtheorem{lemma}[theorem]{Lemma}
\newtheorem{proposition}[theorem]{Proposition}

\theoremstyle{definition}
\newtheorem{definition}[theorem]{Definition}

\theoremstyle{remark}
\newtheorem{remark}[theorem]{Remark}

\numberwithin{equation}{section}

\newcommand{\norm}[1]{\lVert#1\rVert}

\newcommand{\cL}{\mathcal{L}}
\newcommand{\tO}{\widetilde{\Omega}}
\newcommand{\cP}{\mathcal{P}}
\newcommand{\R}{\mathbb{R}}
\newcommand{\N}{\mathbb{N}}
\newcommand{\bS}{\mathbb{S}}
\newcommand{\bH}{\mathbb{H}}
\newcommand{\Z}{\mathbb{Z}}
\newcommand{\T}{\mathbb{T}}
\newcommand{\e}{\varepsilon}
\newcommand{\txt}[1]{\text{#1}}

\begin{document}

\title[Homogenization in domains of finite type]{Homogenization and boundary layers \\in domains of finite type}


\author{Jinping Zhuge}
\address{Department of Math, University of Kentucky, Lexington, KY, 40506, USA.}
\curraddr{}
\email{jinping.zhuge@uky.edu}
\thanks{The author is supported in part by National Science Foundation grant DMS-1161154.}
\date{\today}


\begin{abstract}
	This paper is concerned with the homogenization of Dirichlet problem of elliptic systems in a bounded, smooth domain of finite type. Both the coefficients of the elliptic operator and the Dirichlet boundary data are assumed to be periodic and rapidly oscillating. We prove the theorem of homogenization and obtain an algebraic rate of convergence that depends explicitly on dimension and the type of the domain.
\end{abstract}

\maketitle
\section{Introduction}
In this paper, we consider the oscillating Dirichlet problem for uniformly elliptic systems in general domains,
\begin{equation}\label{eq_Leue}
	\left\{
	\begin{aligned}
		\cL_\e u_\e(x) &= 0 &\quad & \txt{in } \Omega, \\
		u_\e(x) &= f(x,x/\e) &\quad & \txt{on } \partial\Omega,
	\end{aligned}
	\right.
\end{equation}
where
\begin{equation*}
	\cL_\e = -\txt{div} (A(x/\e) \nabla) = - \frac{\partial}{\partial x_i} \bigg\{ a^{\alpha\beta}_{ij} \Big( \frac{x}{\e}\Big) \frac{\partial}{\partial x_j}\bigg\}
\end{equation*}
is a second-order elliptic operator in divergence form with rapidly oscillating periodic coefficients (Einstein's convention for summation will be used throughout). Here $\e>0$ is a small parameter and $\Omega\subset \R^d$ is a smooth, bounded domain without the assumption of strict convexity and $d\ge 2$. We assume that the coefficients $A = A(y) = (a_{ij}^{\alpha\beta})$, with $1\le i,j\le d$ and $1\le \alpha,\beta\le m$, satisfies the ellipticity condition,
\begin{equation}\label{cdn_ellipticity}
	\lambda |\xi|^2 \le a_{ij}^{\alpha\beta} \xi_i^\alpha \xi_j^\beta \le \lambda^{-1}|\xi|^2, \qquad \txt{for any } \xi = (\xi_i^\alpha) \in \R^{m\times d},
\end{equation}
where $\lambda\in (0,1)$ is a fixed constant. Both $A(y)$ and the Dirichlet boundary data $f(x,y)$ are assumed to be 1-periodic in $y$, i.e.,
\begin{equation}\label{cdn_periodicity}
	A(y+z) = A(y) \quad \txt{and} \quad f(x,y+z) = f(x,y) \quad \forall x\in \partial\Omega, y\in \R^d, z\in \Z^d.
\end{equation}
Since we will not try to compute the minimal regularity required for $A$ and $f$, we simply assume that
\begin{equation}\label{cdn_smooth}
	A \in C^\infty(\R^d) \quad \txt{and} \quad f\in C^\infty(\partial\Omega\times\R^d).
\end{equation}

The asymptotic analysis of problem (\ref{eq_Leue}), which is closely related to the higher order convergence rates of homogenization problems with non-oscillating boundary data (see Theorem \ref{thm_high}), was raised in \cite{BLP78}, for instance, and remained open for decades until recently; see \cite{AA99,GerMas11,GerMas12,ASS13,ASS14,ASS15,Alek16,AKMP16,ShenZhuge16,Fe14,Pra13} and references therein. The pioneering work was due to G\'{e}rard-Varet and Masmoudi in \cite{GerMas12}. Under the extra assumption that $\Omega$ is strictly convex, they proved that as $\e \to 0$, the unique solution of (\ref{eq_Leue}) $u_\e$ converges strongly in $L^2(\Omega)$ to some function $u_0$, which is a solution of
\begin{equation}\label{eq_L0u0}
	\left\{
	\begin{aligned}
		\cL_0 u_0(x) &= 0 &\quad & \txt{in } \Omega, \\
		u_0(x) &= \bar{f}(x) &\quad & \txt{on } \partial\Omega,
	\end{aligned}
	\right.
\end{equation}
where the operator $\cL_0$ is given by $\cL_0 = -\txt{div}(\widehat{A}\nabla)$, with $\widehat{A}$ being the usual homogenized matrix of $A$, and $\bar{f}$ is the homogenized Dirichlet boundary data that depends non-trivially on $f,A$ and also $\Omega$ (see \ref{eq_barf}). Moreover, they showed that for each $\delta>0$,
\begin{equation*}
	\norm{u_\e - u_0}_{L^2(\Omega)} \le C \e^{\frac{d-1}{3d+5} - \delta},
\end{equation*}
where $C$ depends on $\delta,d,m,A,f$ and $\Omega$. Most recently, under the same conditions, remarkable improvement was made in \cite{AKMP16} for Dirichlet problems, where the authors obtained nearly optimal convergence rates for $d\ge 4$ and improved suboptimal convergence rates for $d=2,3$. Then soon in \cite{ShenZhuge16}, with new ingredients from $A_p$ weighted estimates, Shen and the author of this paper obtained the nearly optimal convergence rates for Neumann problems with first-order oscillating Neumann boundary data for all dimensions, as well as the lower dimensional cases ($d=2,3$) of Dirichlet problems. Precisely, for Dirichlet problems, it was proved in \cite{AKMP16} and \cite{ShenZhuge16} that
\begin{equation}\label{est_convex_rate}
	\norm{u_\e - u_0}_{L^2(\Omega)} \le 
	\left\{
	\begin{aligned}
		C\e^{\frac{1}{4} - \delta} \qquad & \txt{ for }& d=2, \\
		C \e^{\frac{1}{2} - \delta} \qquad & \txt{ for } & d\ge 3.
	\end{aligned}
	\right.
\end{equation}
and the homogenized data $\bar{f}\in W^{1,q}\cap L^\infty(\partial\Omega)$ for any $q<d-1$. The convergence rates in (\ref{est_convex_rate}) are optimal up to an arbitrarily small loss on the exponent, for the optimal convergence rates, i.e., $O(\e^{1/4})$ for $d=2$ and $O(\e^{1/2})$ for $d\ge 3$, were shown in \cite{ASS15} for operators with constant coefficients. The regularity for $\bar{f}$ is also sharp in a certain sense as shown in \cite{ShenZhuge16}.

We emphasize that the previous results and their proofs rely essentially on the geometry of the boundary while domains with a different geometry will have completely different behaviors, for example, the difference between strictly convex domains \cite{GerMas12,ASS15} and polygon domains \cite{GerMas11,ASS14}. Also in a very recent paper \cite{Alek16}, it was shown constructively that the rate of convergence for (\ref{eq_Leue}) can be arbitrarily slow if the domain is merely \emph{non-strictly} convex. In fact, the strict convexity of the domain plays an essential role in the arguments of \cite{GerMas12,AKMP16,ShenZhuge16}, for which we will give a brief interpretation in two aspects as follows. On one hand, near a given point on $\partial\Omega$, the local behavior of the solution depends whether or not the direction of the normal vector to $\partial\Omega$ is non-resonant (with the lattice $\Z^d$). If it is non-resonant, the quantitative analysis depends further on the so called Diophantine condition of the normal vector. For this reason, the regularity of the Diophantine function on the boundary, which is definitely guaranteed by the strict convexity of a domain, does have a crucial impact on the rates of convergence. On the other hand, when approximating $u_\e$, a solution of an oscillating Dirichlet problem, near a given point on $\partial\Omega$ by a solution of a half-space problem with the corresponding tangent plane being the boundary, the convexity of $\Omega$ ensures that the domain lies on one side of the tangent plane and hence $u_\e$ could be well approximated near the given point. This fact obviously fails if $\Omega$ is a non-convex domain.

The main purpose of this paper is to remove the assumption of strict convexity and obtain an algebraic rate of convergence that depends explicitly on certain quantitative property of the boundary. We should point out that this is not the first paper dealing with oscillating boundary value problems in general domains. In \cite{Fe14}, the authors established the qualitative homogenization for Dirichlet problems of (scalar) fully nonlinear uniformly elliptic operators in general domains, whose irrational normals have a small Hausdorff dimension. In \cite{CK14} and \cite{FKS15}, the same type of equations with oscillating Neumann boundary data was studied in general domains without flatness. However, as far as we know, this should be the first paper that studies the quantitative homogenization of periodic oscillating boundary value problem in general non-convex smooth domains.

Though the assumption of strict convexity on the domains will be removed, as we have mentioned before, we do require some restrictive condition to rule out any boundary with nontrivial resonant portion, for example, a domain with flat portion with rational normal on the boundary. In this paper, we introduce a mild condition on $\partial\Omega$ to characterize the non-flatness of $\partial\Omega$, namely, the hypersurfaces of finite type. The notion of finite type has been well studied in many references mainly in Fourier analysis and could be defined geometrically in terms of the condition of finite order of contact on the principle curvatures. For our convenience in this paper, we prefer to introduce an equivalent analytical definition. We first give a definition of functions of finite type; also see \cite[VIII.2]{Stein93}.
\begin{definition}\label{def_f}
	A smooth function $g$ is of type $k$ ($k \ge 2$) in some connected open set $U$, if there exist some multi-index $\alpha$ with $1<|\alpha| \le k$ and $\delta > 0$ such that $|\partial^\alpha g| \ge \delta$ in $U$.
\end{definition}
Let $S$ be a smooth hypersurface in $\R^d$. For any $x_0\in S$, we can translate and rotate the hypersurface so that $x_0$ is moved to the origin and the tangent plane becomes $x_d = 0$. Meanwhile, near $x_0$, $S$ is transformed to a graph of some function $x_d = \phi = \phi_{x_0}:\R^{d-1} \to \R$, which satisfies $\phi(0) = 0$ and $\nabla\phi (0) = 0$. We call $x_d = \phi(x') = \phi_{x_0}(x')$ the \emph{local graph} of $S$ at $x_0$.

\begin{definition}\label{def_Omg}
	A smooth hypersurface $S\subset \R^d$ is called type $k$, if for each $x_0 \in S$, there exist constants $r>0$ and $\delta>0$ such that the corresponding local graph $x_d = \phi_{x_0}(x')$ is of type $k$ in $B(0,r) \subset \R^{d-1}$ with a lower bound $\delta$. For a smooth domain $\Omega$, we also say $\Omega$ is of type $k$ if its boundary $\partial\Omega$ is of type $k$.\footnote{If $S$ is a closed smooth hypersurface (i.e., the boundary of a bounded smooth domain), the constants $r$ and $\delta$ can be chosen uniformly with respect to $x_0$. The condition of finite type on a hypersurface sometimes is referred as finite order of contact with the tangent plane.}
\end{definition}

Now, we state the main theorem as follows.

\begin{theorem}\label{thm_main}
	Assume that $\Omega$ is a bounded, smooth domain of type $k$, and that (\ref{cdn_ellipticity}), (\ref{cdn_periodicity}) and (\ref{cdn_smooth}) hold. Then problem (\ref{eq_Leue}) admits homogenization, i.e., the solution $u_\e$ of (\ref{eq_Leue}) converges to $u_0$ in $L^2(\Omega)$, where $u_0$ is the unique solution of the corresponding homogenized problem (\ref{eq_L0u0}) with boundary data $\bar{f}$. Moreover, there exist some $q^*>0$ and $\alpha^* > 0$ depending explicitly on $k$ and $d$ such that
	\begin{equation}\label{eq_barf_W1q}
	\bar{f} \in W^{1,q}\cap L^\infty(\partial\Omega), \qquad \txt{for any } q<q^*,
	\end{equation}
	and
	\begin{equation}\label{est_ueu0L2}
		\norm{u_\e - u_0}_{L^2(\Omega)} \le C \e^{\frac{1}{2}\alpha^* - \delta},
	\end{equation}
	where $\delta>0$ is an arbitrarily small number and $C$ depends on $\delta,k, d,m,A,f$ and $\Omega$.
\end{theorem}

The explicit formula for $q^*$ is given by
\begin{equation}\label{eq_q*}
q^* = \frac{d-1}{2\gamma-1},
\end{equation}
with $\gamma = (d-1)(k-1)$. While the explicit formula for $\alpha^*$, which comes from optimizing several error terms, is given by
\begin{equation}\label{eq_alpha*}
	\alpha^* = \max_{s\in [1/2,1]} \bigg[ s \wedge 4(s - \frac{1}{2}) \wedge (d-1)(s - \frac{1}{2})
	\wedge \frac{(1-s)(d-1)}{\gamma -1} \wedge \frac{s(d-1)}{1+\gamma}\bigg],
\end{equation}
where $\gamma$ is the same and $a\wedge b := \min\{a, b\}$. Observe that given specific $d$ and $k$, $\alpha^*$ can be computed by solving a linear programming problem. We remark that except for lower dimensions and $k=2$ (see Remark \ref{rmk_finite}), the formula for $\alpha^*$ seems suboptimal in general, in view of the exponent obtained for non-oscillating operators ($\alpha^* = 1/k$, regardless of dimensions; see Theorem \ref{thm_const} for details). Nevertheless, finding a better or optimal rate for general $k$ and $d$ should also be an interesting problem.

The first reduction of the proof of Theorem \ref{thm_main} is due to a new idea for quantifying the non-flatness of the boundary. Recall that one of the key points in oscillating boundary value problems is the so called Diophantine condition; see Definition \ref{def_Dioph}. It has been shown in \cite{GerMas12} that if $\Omega$ is a bounded, smooth and strictly convex domain, then the reciprocal of \emph{Diophantine function} $\varkappa \circ n(x) = \varkappa(n(x))$ is in $L^{d-1,\infty}(\partial\Omega,d\sigma)$, where $\sigma$ is the surface measure of $\partial\Omega$ and $n(x)$ is the unit outer normal to $\partial\Omega$ at $x$. This condition played a key role in \cite{GerMas12,AKMP16,ShenZhuge16} for oscillating boundary value problems in strictly convex domains. Similarly in our setting, to prove Theorem \ref{thm_main}, we would like to show that if $S$ is a closed (i.e., compact and boundaryless) smooth hypersurface of finite type, then $(\varkappa\circ n)^{-1} \in L^{p,\infty}(S,d\sigma)$ for some $p>0$ depending explicitly on the type (actually, $p=1/(k-1)$ if $k$ is the type). Surprisingly, we are able to show that $S$ is finite type if and only if there exists some $p>0$ such that $(\varkappa\circ On)^{-1} \in L^{p,\infty}(S,d\sigma)$ uniformly for any orthogonal matrix $O$. This equivalence is a consequence of the famous van de Corput's lemma and a criteria established in Proposition \ref{prop_weakp} which involves the so called sublevel set estimate. The orthogonal matrix involved here is a necessary requirement for rotation invariance of the finite type assumption. As another easy corollary of the criteria, we also show that $S$ is strictly convex if and only if $(\varkappa\circ On)^{-1} \in L^{d-1,\infty}(S,d\sigma)$ uniformly for all orthogonal matrix $O$. In addition, the relationships between finite type condition, oscillatory integrals, the property of Diophantine function and homogenization are summarized in Figure \ref{fig_1}.

As a consequence, Theorem \ref{thm_main} is reduced to

\begin{theorem}\label{thm_maina}
	Assume that $\Omega$ is a bounded, smooth domain such that $(\varkappa\circ n)^{-1} \in L^{p,\infty}(\partial\Omega,d\sigma)$, 
	and let $A$ and $f$ be the same as Theorem \ref{thm_main}. Then (\ref{eq_Leue}) admits homogenization, and (\ref{eq_barf_W1q}) and (\ref{est_ueu0L2}) hold with $q^*$ and $\alpha^*$ given by (\ref{eq_q*}) and (\ref{eq_alpha*}), respectively, where $\gamma = (d-1)/p$.
\end{theorem}

Notice that Theorem \ref{thm_maina} is an even more general result which recovers the 
special case of strictly convex domains. Actually, if $\Omega$ is bounded, smooth and strictly convex, then $(\varkappa\circ n)^{-1} \in L^{d-1,\infty}(\partial\Omega,d\sigma)$ and hence $\gamma = 1$ in (\ref{eq_q*}) and (\ref{eq_alpha*}). Thus, one sees that $q^* = d-1$ and $\alpha^* = 1\wedge \frac{d-1}{2}$, which exactly coincides with the exponents of (\ref{est_convex_rate}) for all dimensions $d\ge 2$. For $p<d-1$, it is unknown whether $q^*$ or $\alpha^*$ are optimal.

We now present a sketch of the proof of Theorem \ref{thm_maina}. Our approach follows the line of \cite{AKMP16} and \cite{ShenZhuge16} with some technical modifications. The starting point is the Poisson integral formula for $u_\e$ and an expansion for Poisson kernel established in \cite{AveLin87},
\begin{equation}\label{eq_ue_exp}
	\begin{aligned}
		u_\e(x) & = \int_{\partial\Omega} P_{\Omega,\e}(x,y) f\Big(y,\frac{y}{\e}\Big) d\sigma(y)\\
		& = \int_{\partial\Omega} P_{\Omega}(x,y) \omega_{\e}(y) f\Big(y,\frac{y}{\e}\Big) d\sigma(y) + \txt{ small error}
	\end{aligned}
\end{equation}
where $P_{\Omega,\e}$ and $ P_{\Omega}$ are the Poisson kernels associated with $\cL_\e$ and $\cL_0$, respectively, in $\Omega$, and $\omega_\e$ is an oscillating function correcting the Poisson kernel. The integral in the second line of (\ref{eq_ue_exp}) will be denoted by $\tilde{u}_\e$ and it is sufficient to consider the $L^2$ error of $\tilde{u}_\e - u_0$. In view of (\ref{eq_omgExp}), the only obscure factor in $\omega_\e$ is $\nabla \Phi^*_\e$, the adjoint Dirichlet correctors introduced in \cite{KLS14}. We then show that in a neighborhood of some given point on $\partial\Omega$, $\nabla \Phi^*_\e$ can be approximated by a function in the form of
\begin{equation}\label{eq_dPhi*}
	\nabla \Phi^*_\e(x) \sim I + \nabla\chi^*(x/\e) + \nabla \bar{v}^*_\e(x)
\end{equation}
where $I$ is the identity matrix and $\chi^*$ is the usual adjoint corrector. In the case of convex domains, the third term on the right-hand side of (\ref{eq_dPhi*}) is a solution of a half-space problem depending on the given point, as $\Omega$ is on one side of the tangent plane. However, for non-convex domains in our setting, $\Omega$ may lie on both sides of the tangent plane. To handle this situation, we first obtain the approximation (\ref{eq_dPhi*}) in a sub-domain of $\Omega$ which inscribes $\partial\Omega$ at the given point and lies on one side of the tangent plane. Then we extend (\ref{eq_dPhi*}) to the other side by a standard extension argument with $C^2$ regularity preserved. It turns out that the convexity actually plays no role in the approximation (\ref{eq_dPhi*}).

We can then proceed as in \cite{ShenZhuge16} by introducing a Carder\'{o}n-Zygmund-type decomposition of the boundary adapted to the function $F = \varkappa^{1/\gamma}$, where $\gamma$ is given as in Theorem \ref{thm_maina}, and then for each small surface cube in the decomposition, approximating the integral on $\partial\Omega$ by the integral on a carefully selected tangent plane. The sizes of surface cubes in the decomposition sit in $(\tau,\sqrt{\tau})$, where $\tau$ is a small parameter related to $\e$. Unlike the case of strictly convex domains where $\tau = \e^{1-}$ is fixed, for general domains considered in Theorem \ref{thm_maina}, we have to set $\tau = \e^s$ for some $s\in (1/2,1)$ to be determined and eventually identify the best $s$ by optimizing several error terms. Actually, with the decomposition proposed, we split the $L^2$ error of $\tilde{u}_\e - u_0$ into two parts. The trivial part is the $L^2$ estimate of $\tilde{u}_\e - u_0$ over a small boundary layer $\Gamma_\e$ whose volume is $O(\e^s)$. The interior part is much involving and splits further into 5 errors ($I_k, k=1,2,3,4,5$), which are contained in the following:
\begin{equation*}
	\begin{aligned}
		\txt{Total error of }u_\e - u_0 = & \txt{ Error } E_1 \txt{ coming from (\ref{eq_ue_exp})} \\
		&+ \txt{Error } E_2 \txt{ coming from trivial estimate in } \Gamma_\e \\
		&+ \txt{Error } I_1 \txt{ coming from (\ref{eq_dPhi*})} \\
		&+ \txt{Error } I_2 \txt{ coming from projection onto a tangent plane} \\
		&+ \txt{Error } I_3 \txt{ coming from quantitative ergodic theorem} \\
		&+ \txt{Error } I_4 \txt{ coming from regularity of homogenized boundary data} \\
		&+ \txt{Error } I_5 \txt{ coming from changing variables back to } \partial\Omega.
	\end{aligned}
\end{equation*}
The right-hand side of the last expression includes all the error terms needed to bound the total $L^2$ error of $u_\e - u_0$. We point out that the error $E_1$ is $O(\e^{1-})$; the error $E_2$ contributes the exponent $s$ in (\ref{eq_alpha*}); $I_1$ and $I_2$ contribute to exponent $4(s - \frac{1}{2}) \wedge (d-1)(s - \frac{1}{2})$; $I_3$ contributes to exponent $\frac{(1-s)(d-1)}{\gamma -1}$; $I_4$ contributes to exponent $\frac{s(d-1)}{1+\gamma}$ and $I_5$ is smaller and ignored. Note that the parameter $\gamma$, which quantifies the non-flatness of the boundary, is only involved in $I_3$ and $I_4$, due to the presence of the Diophantine function in the arguments. We should also mention that the quantitative ergodic theorem for the estimate of $I_3$ was introduced in \cite{AKMP16} and the (optimal) regularity of the homogenized data involved in the estimate of $I_4$ was first proved in \cite{ShenZhuge16} via an $A_p$ weighted estimate.

The outline of the paper is described as follows. The preliminaries are given in Section 2, including standard notations used throughout. Section 3 is devoted to the quantification of the non-flatness of hypersurfaces of finite type. In Section 4, we establish the approximation (\ref{eq_dPhi*}). In Section 5, we introduce a partition of unity on $\partial\Omega$ with the assumption $\varkappa(\cdot)^{-1} \in L^{p,\infty}(\partial\Omega)$ for some $p>0$. Finally, Theorem \ref{thm_maina} and hence Theorem \ref{thm_main} are proved in Section 6, as well as a theorem concerning the higher order convergence rate for non-oscillating boundary value problems.

{\bf Acknowledgment.} The author would like to thank professor Zhongwei Shen for helpful discussions and suggestions. 

\section{Preliminaries}
Throughout we will use the following notations. The indices $i,j,k,\ell$ usually denote integers ranging between $1$ and $d$, whereas the small Greek letters $\alpha,\beta,\gamma$ usually denote integers ranging between $1$ and $m$. The vector $e_i \in \R^d$ stands for $i$-th vector of canonical basis of $\R^d$ and $e^\alpha \in \R^m$ stands for $\alpha$-th vector in canonical basis of $\R^m$. For $x = (x_1,x_2,\cdots,x_d) \in \R^d$, we write $x = (x',x_d)$ where $x'\in \R^{d-1}$. We use $\bS^{d-1}$ and $\T^d$ to denote the $d-1$ dimensional unit sphere and $d$ dimensional torus, respectively. For the coefficient matrix $A$, we write $A_\e(x) = A(x/\e)$ for simplicity.

For real numbers $a,b$, we let $a\wedge b = \min\{ a,b\}$ and $a\vee b = \max \{a,b\}$. As usual, $C$ and $c$ are positive constants that may vary from line to line and they depend at most on $d,m,A$ and $\Omega$ as well as other parameters, but never on $\e$ or the Diophantine constant $\varkappa$. The dependence on $A$ should be interpreted as both the ellipticity constant $\lambda$ and $\norm{A}_{C^k(\T^d)}$ for some $k=k(d)>1$. In Section 6, we use $\delta$ to denote an arbitrarily small exponent which may also differ in each occurrence.

Assume that $A$ satisfies (\ref{cdn_ellipticity}) and (\ref{cdn_periodicity}). For each $1\le j\le d, 1\le \beta\le m$, let $\chi = (\chi_j^\beta) = (\chi_j^{1\beta},\chi_j^{2\beta},\cdots,\chi_j^{m\beta})$ denote the correctors for $\cL_\e$, which are 1-periodic functions satisfying the cell problem
\begin{equation*}
	\left\{
	\begin{aligned}
		\cL_1 (\chi^{\beta}_{j} + P_j^\beta)(x) &= 0 \qquad  \txt{ in } \T^d, \\
		\int_{\T^d} \chi_j^\beta & = 0,
	\end{aligned}
	\right.
\end{equation*}
where $P_j^\beta(x) = x_je^\beta $. Note that $\nabla P_j^\beta = e_j e^\beta$.

We introduce the matrix of Dirichlet boundary correctors $\Phi_\e = \Phi_{\e,j}^\beta = (\Phi_{\e,j}^{1\beta},\Phi_{\e,j}^{2\beta},\dots,\Phi_{\e,j}^{m\beta})$ associated with $\cL_\e$ in a bounded domain $\Omega$. Indeed, for each $1\le j\le d, 1\le \beta\le m$, $\Phi_{\e,j}^{\beta}$ is the solution of
\begin{equation*}
	\left\{
	\begin{aligned}
		\cL_\e \Phi^{\beta}_{\e,j}(x) &= 0 \qquad & \txt{ in }& \Omega, \\
		\Phi^{\beta}_{\e,j}(x) &= P_j^\beta (x) \qquad & \txt{ on } &\partial\Omega.
	\end{aligned}
	\right.
\end{equation*}

The homogenized operator is given by $\cL_0 = -\txt{div}(\widehat{A}\nabla)$, where the homogenized matrix $\widehat{A} = (\widehat{a}_{ij}^{\alpha\beta})$ is defined by
\begin{equation*}
	\widehat{A} = \int_{\T^d} A(I+\nabla \chi) \quad \txt{or in component form} \quad \widehat{a}_{ij}^{\alpha\beta} = \int_{\T^d} \bigg\{ a_{ij}^{\alpha\beta} + a_{ik}^{\alpha\gamma} \frac{\partial}{\partial y_k} (\chi_j^{\gamma\beta}) \bigg\}.
\end{equation*}
We also introduce the adjoint operator $\cL^*_\e = - \txt{div}(A^*_\e \nabla )$, where $A^* = (a_{ij}^{*\alpha\beta})$ and $a_{ij}^{*\alpha\beta} = a_{ji}^{\beta\alpha}$. Note that $A^*$ also satisfies (\ref{cdn_ellipticity}), (\ref{cdn_periodicity}) and (\ref{cdn_smooth}). Let $\chi^*$ and $\Phi_{\e}^*$ be the adjoint correctors and the adjoint Dirichlet boundary correctors, respectively, associated with $\cL^*$.

Let $\Omega$ be a bounded $C^{2,\sigma}$ domain and $\sigma \in (0,1)$. The matrix of Poisson kernel $P_{\Omega,\e}: \Omega\times\partial\Omega \mapsto \R^{m\times m}$, associated with $\cL_\e$ in $\Omega$, is defined by
\begin{equation*}
	P_{\Omega,\e}^{\alpha\beta}(x,y) = -n(y)\cdot a^{\gamma\beta}(y/\e) \nabla_y G_{\Omega,\e}^{\alpha\gamma}(x,y),
\end{equation*}
where $n(y)$ is the unit outer normal and $G_{\Omega,\e}$ is the matrix of Green's function associated with $\cL_\e$ in $\Omega$. The following uniform estimates in \cite{AveLin87} will be useful,
\begin{equation}\label{est_Poi1}
	|P_{\Omega,\e}(x,y)| \le \frac{C}{|x-y|^{d-1}},
\end{equation}
and
\begin{equation}\label{est_Poi2}
	|P_{\Omega,\e}(x,y)| \le \frac{C\txt{dist}(x,\partial\Omega)}{|x-y|^{d}}.
\end{equation}
Let $P_{\Omega}$ be the Poisson kernel associated with the homogenized operator $\cL_0$ in $\Omega$. Clearly, $P_{\Omega}$ possesses the same estimates (\ref{est_Poi1}) and (\ref{est_Poi2}).

Recall that the two-scale expansion of the Poisson kernel of $\cL_\e$ in $\Omega$ was established in \cite{KLS14},
\begin{equation}\label{eq_Poisson_exp}
	P_{\Omega,\e}^{\alpha\beta}(x,y) = P_{\Omega}^{\alpha\gamma}(x,y) \omega_\e^{\gamma\beta}(y) + R_\e^{\alpha\beta}(x,y) \qquad \txt{for } x\in \Omega, y\in \partial\Omega,
\end{equation}
where $R_\e$ is the remainder term satisfying
\begin{equation*}
	|R_\e(x,y)| \le \frac{C \e \ln(2+\e^{-1} |x-y|)}{|x-y|^d}.
\end{equation*}
The highly oscillating factor $\omega_\e(y)$ in (\ref{eq_Poisson_exp}) is given by
\begin{equation}\label{eq_omgExp}
	\omega_\e^{\gamma\beta}(y) =  h^{\gamma\nu }(y) \cdot n_k(y) n_{\ell}(y) \frac{\partial }{\partial y_\ell} \Phi_{\e,k}^{*\rho\nu}(y) \cdot a_{ij}^{\rho\beta}(y/\e) n_i(y)n_j(y),
\end{equation}
and $h(y)$ is the inverse matrix of $\widehat{a}_{ij}(y) n_i(y) n_j(y)$.

Let $u_\e$ be the solution of (\ref{eq_Leue}). By Poisson integral formula, we have
\begin{equation*}
	u_\e(x) = \int_{\partial\Omega} P_{\Omega,\e}(x,y) f(y,y/\e) d\sigma(y).
\end{equation*}
Note that (\ref{est_Poi2}) implies the Agmon-type maximum principle $\norm{u_\e}_{L^\infty(\Omega)} \le C\norm{f}_{L^\infty(\partial\Omega\times \T^d)}$, which we will often refer to. Define
\begin{equation*}
	\tilde{u}_\e(x) = \int_{\partial\Omega} P_{\Omega}(x,y) \omega_\e(y) f(y,y/\e) d\sigma(y).
\end{equation*}

\begin{lemma}\label{lem_uetue}
	Let $\Omega$ be a bounded $C^{2,\sigma}$ domain and let (\ref{cdn_ellipticity}), (\ref{cdn_periodicity}) and (\ref{cdn_smooth}) hold. Then
	\begin{equation*}
		\norm{u_\e - \tilde{u}_\e}_{L^q} \le C\e^{1/q}(1+|\ln\e|) \norm{f}_{L^\infty(\partial\Omega\times \T^d)}.
	\end{equation*}
	for any $1\le q<\infty$.
\end{lemma}
This follows readily from (\ref{eq_Poisson_exp}) and a similar proof can be found in \cite[Lemma 2.3]{ShenZhuge16}. Thanks to Lemma \ref{lem_uetue}, the estimate for $\norm{u_\e - u_0}_{L^2(\Omega)}$ is reduced to $\norm{\tilde{u}_\e - u_0}_{L^2(\Omega)}$.

\section{Quantification of hypersurfaces of finite type}
\subsection{Quantification of non-flatness}
In this section, we develop a new idea involving the Diophantine condition to quantify the non-flatness of hypersurfaces of a domain which could be used in homogenization of oscillating boundary value problems. It turns out that this idea has a close connection with the sublevel set estimate and van de Corput's lemma in the theory of oscillatory integrals.

To begin with, we recall the Diophantine condition for a unit vector $n \in \bS^{d-1}$.

\begin{definition}\label{def_Dioph}
	Given $n\in \bS^{d-1}$. We say $n$ satisfies the Diophantine condition with some fixed $\mu>0$, if there exists some constant $C > 0$ such that\footnote{The matrix $I-n\otimes n$ is the orthogonal projection onto $n^\perp$.}
	\begin{equation}\label{cdn_Diophantine}
		|(I - n\otimes n)\xi| \ge C |\xi|^{-\mu} \qquad \txt{for all } \xi\in \Z^d\setminus \{0\}.
	\end{equation}
	We call $\varkappa = \varkappa(n)$ the Diophantine constant if it is the largest constant such that (\ref{cdn_Diophantine}) holds. As a function of $n$, $\varkappa(n)$ will also be called the Diophantine function on $\bS^{d-1}$.
\end{definition}

One can slightly generalize the Diophantine function from $\bS^{d-1}$ to any closed smooth hypesurface $S$ by considering $ \varkappa\circ n(x)$ for all $x\in S$, where $n(x)$ is the outer normal to $S$ at $x$. As we have mentioned in Introduction, the fact $(\varkappa\circ n)^{-1} \in L^{d-1,\infty}(\partial\Omega)$ played a crucial role in oscillating boundary value problems in strictly convex domains considered in \cite{GerMas12,AKMP16,ShenZhuge16}. For general smooth compact hypersurfaces, one might also ask if there is still some possible $p$ less than $d-1$ such that $(\varkappa\circ n)^{-1} \in L^{p,\infty}(S,d\sigma)$ and what condition would exactly guarantee this. In this section, we will give positive answers to these two questions.

As stated in Introduction, for any $x_0\in S$, we can translate and rotate the hypersurface so that $S$ is given by its local graph $x_d = \phi(x')  = \phi_{x_0}(x')$ near $x_0$.

\begin{proposition}\label{prop_weakp}
	Let $S$ be a closed smooth hypersurface in $\R^d$. The following statements are equivalent:
	
	(i) there is some $\mu>0$ so that $(\varkappa \circ On(\cdot))^{-1} \in L^{p,\infty}(S,d\sigma)$ uniformly for any orthogonal matrix $O$;
	
	(ii) the function $h_{\omega}$ defined below satisfies
	\begin{equation}\label{cdn_h_weakp}
		h_\omega(x) := \frac{1}{\sqrt{1-[\omega\cdot n(x)]^2}} \in L^{p,\infty}(S,d\sigma),
	\end{equation}
	uniformly for any $\omega\in \bS^{d-1}$;
	
	(iii) there exist some $r_0>0$ and $C_0>0$ such that for all $x_0 \in S$, its local graph satisfies
	\begin{equation}\label{est_Criteria}
		\sigma\{ x'\in B(0,r_0) \subset \R^{d-1}: |\nabla \phi_{x_0}(x')| \le t \} \le C_0t^p.
	\end{equation}
	for all $0<t<1$.
\end{proposition}

\begin{proof}
	First, we show that (i) is equivalent to (ii). We assume the statement (i) holds. For a fixed $\omega \in \bS^{d-1}$, we can find an orthogonal matrix $O$ such that $\omega$ is the first column of $O^t$, where $O^t$ is the transpose of $O$. This implies that $\omega \cdot n = e_1\cdot On$. Now a key observation is 
	\begin{equation}\label{cdn_h_til}
		\sqrt{1-[\omega\cdot n]^2} = |(I-n\otimes n) \omega|, \quad \forall \omega, n\in \bS^{d-1}.
	\end{equation}
	It follows
	\begin{equation*}
		\sqrt{1-[\omega\cdot n]^2} = \sqrt{1-[e_1\cdot On]^2} = |(I-On\otimes On) e_1| \ge \varkappa\circ On |e_1|^{-\mu}.
	\end{equation*}
	As a result
	\begin{equation*}
		\sigma \{x\in S: \sqrt{1-[\omega\cdot n(x)]^2} < t\} \le \sigma \{x\in S: \varkappa\circ On(x) < t \} \le Ct^p,
	\end{equation*}
	where in the second inequality we used the condition $(\varkappa\circ On)^{-1}\in L^{p,\infty}(S,d\sigma)$. Clearly, this implies that $h_\omega$ is in $L^{p,\infty}(S,d\sigma)$, uniformly.
	
	Now we assume that the statement (ii) holds.
	Let $0<t<1$. Observe that
	\begin{equation}
		\{x\in\partial\Omega: (\varkappa\circ O n(x))^{-1} > t^{-1}\} \subset S_{t} = \bigcup_{\xi \in \Z^d \setminus\{ 0 \} } \{ x\in \partial\Omega: |(I-On\otimes On)\xi| < t |\xi|^{-\mu} \}.
	\end{equation}
	Using (\ref{cdn_h_weakp}) and (\ref{cdn_h_til}), we have
	\begin{align*}
		&\sigma\{x\in \partial\Omega: |(I-On\otimes On)\xi| < t |\xi|^{-\mu} \} \\
		&\qquad = \sigma\{x\in \partial\Omega: |(I-On\otimes On)\omega| < t |\xi|^{-1-\mu}, \omega = |\xi|^{-1}\xi \} \\
		&\qquad = \sigma\{x\in \partial\Omega: \sqrt{1-[O^t\omega\cdot n(x)]^2} < t |\xi|^{-1-\mu}, \omega = |\xi|^{-1}\xi \} \\
		& \qquad \le Ct^{p} |\xi|^{-p(1+\mu)}.
	\end{align*}
	Now we choose $\mu$ sufficiently large so that 
	$p(1+\mu)> d$. Then it follows
	\begin{equation*}
		\sigma \{x\in\partial\Omega: (\varkappa \circ n(x))^{-1} > t^{-1}\} \le \sigma(S_t) \le C t^p,
	\end{equation*}
	for any $0<t<1$. This finishes the proof of equivalence between (i) and (ii).
	
	Next we show that (ii) is equivalent to (iii). Assume that (ii) is true. Let $x_0$ be a point on $S$ and $x_d = \phi(x') = \phi_{x_0}(x')$ be the local graph of $S \cap B(x_0,r_0)$, where $r_0$ is given in Definition \ref{def_Omg}. Recall that by definition, $\phi(0) = 0$ and $\nabla \phi(0) = 0$. Note that in local coordinates
	\begin{equation}\label{eq_locn}
		n(x') = \frac{(\nabla\phi(x'),-1)}{\sqrt{1+|\nabla\phi(x')|^2}}.
	\end{equation}
	Put $w = e_d$. Then, observe that
	\begin{equation}\label{eq_nw_dphi}
		1 - (n(x')\cdot \omega)^2 = \frac{|\nabla \phi(x')|^2 }{1+|\nabla\phi(x')|^2}.
	\end{equation}
	It follows from (ii) that
	\begin{equation*}
		\sigma\{ x'\in B(0,r_0): |\nabla \phi(x')| < t \} \le \sigma\{ x'\in B(0,r_0): \sqrt{1-(n(x')\cdot \omega)^2} < t \} \le Ct^p.\\
	\end{equation*}
	
	On the contrary, we now assume that the statement (ii) is false, which means that for any large $M>0$, there exist $w$ and $t$ (depending on $M$) such that
	\begin{equation*}
		\sigma \{x\in S: \sqrt{1-[\omega\cdot n(x)]^2} < t \} \ge Mt^p.
	\end{equation*}
	Fix such $w$ and $t$. Let $r_0$ be given by Definition \ref{def_Omg}. Then it is not hard to see that there exists some point $x_0 \in S$ such that
	\begin{equation}\label{est_1-nx0}
		\sqrt{1-[\omega\cdot n(x_0)]^2} < t,
	\end{equation}
	and
	\begin{equation*}
		\sigma \{x\in S\cap B(x_0,r_0): \sqrt{1-[\omega\cdot n(x)]^2} < t \} \ge C^{-1}Mt^p,
	\end{equation*}
	where $C$ depends only on $S$. Without loss of generality, we may assume $\omega\cdot n(x) \ge 0$ for $x$ in the set involved above. Now we have a simple observation: for any $u,v\in \bS^{d-1}$ and $u\cdot v \ge 0$,
	\begin{equation*}
		\sqrt{1-(u\cdot v)^2} \le |u-v| \le 2\sqrt{1-(u\cdot v)^2}.
	\end{equation*}
	As a result,
	\begin{equation*}
		\begin{aligned}
			\sqrt{1-(n(x_0)\cdot n(x))^2} & \le |u(x_0)-n(x)| \\
			& \le |u(x_0)-\omega| + |\omega-n(x)| \\
			& \le 2\sqrt{1-(n(x_0)\cdot \omega)^2} + 2\sqrt{1-(\omega\cdot n(x))^2} \\
			& < 4t,
		\end{aligned}
	\end{equation*}
	for all $x\in \{y\in S\cap B(x_0,r_0): \sqrt{1-[\omega\cdot n(y)]^2} < t \}$, where we also used (\ref{est_1-nx0}) in the last inequality. It follows that
	\begin{equation}\label{est_1-n0n}
		\begin{aligned}
			&\sigma \{x\in S\cap B(x_0,r_0): \sqrt{1-[n(x_0)\cdot n(x)]^2} < 4t \} \\
			&\qquad\qquad \ge \sigma \{x\in S\cap B(x_0,r_0): \sqrt{1-[\omega\cdot n(x)]^2} < t \} \ge C^{-1}M t^p.
		\end{aligned}
	\end{equation}
	Again, we apply the local graph $x_d = \phi(x') = \phi_{x_0}(x')$ at $x_0$ and use (\ref{eq_locn}) and (\ref{eq_nw_dphi}) to obtain
	\begin{equation*}
		\begin{aligned}
			&\sigma \{x\in S\cap B(x_0,r_0): \sqrt{1-[n(x_0)\cdot n(x)]^2} < 4t \} \\
			& \qquad \le C \sigma \{x'\in B(x_0,r_0): \sqrt{1-[e_d \cdot n(x')]^2} < 4t \} \\
			& \qquad \le C \sigma \{x'\in B(x_0,r_0):  |\nabla \phi(x')| < Ct \}.
		\end{aligned}
	\end{equation*}
	This, together with (\ref{est_1-n0n}), leads to
	\begin{equation*}
		\sigma \{x'\in B(x_0,r_0):  |\nabla \phi(x')| < Ct \} \ge C^{-1}Mt^p.
	\end{equation*}
	Since $M$ can be arbitrarily large, the last inequality contradicts to (iii). This completes the proof.
\end{proof}

A straightforward application of Proposition \ref{prop_weakp} is to verify that every strictly convex domain satisfies the statements in Proposition \ref{prop_weakp} with $p = d-1$. Actually, they are even sufficient and necessary condition of each other.

\begin{proposition}\label{prop_convex_Lp}
	A closed smooth hypersurface in $\R^d$ is strictly convex if and only if there is some $\mu>0$ so that $(\varkappa \circ On(\cdot))^{-1} \in L^{d-1,\infty}(S,d\sigma)$ uniformly for any orthogonal matrix $O$.
\end{proposition}
\begin{proof}
	First we assume that $S$ is strictly convex. It is sufficient to verify (iii) of Proposition \ref{prop_weakp}. For any fixed $x_0\in S$, let $\phi_{x_0}$ be the local graph of $S$ at $x_0$. Since $S$ is strictly convex, the Hessian matrix $\nabla^2 \phi_{x_0}(0)$ is positive definite and its eigenvalues are bounded below uniformly in $x_0$. It follows from the mean value theorem that
	\begin{equation*}
		|\nabla \phi_{x_0}(x')| = |\nabla^2 \phi_{x_0}(\xi) x' | \ge c|x'|,
	\end{equation*}
	if $|x|<r_0$ for some sufficiently small $r_0$ depending only on $S$. Hence
	\begin{equation*}
		\sigma\{ x'\in B(0,r_0): |\nabla \phi_{x_0}(x')| < t \} \le \sigma\{ x'\in B(0,r_0): |x'| < Ct \} \le Ct^{d-1}.
	\end{equation*}
	This proves the statement (iii) with $p = d-1$.
	
	On the contrary, suppose that $S$ is not strictly convex, then there must be some point $x_0\in S$ such that the Hessian matrix $\nabla^2 \phi_{x_0}(0)$ is degenerate. Without loss of generality, we can rotate the coordinates such that $\nabla^2 \phi_{x_0}(0)$ is diagonal, i.e.,
	\begin{equation*}
		\nabla^2 \phi_{x_0}(0) =
		\begin{bmatrix}
			\lambda_1 & & &\\
			& \ddots & &\\
			& & \lambda_{d-2} &\\
			& & & 0 
		\end{bmatrix}.
	\end{equation*}
	Now writing $x' = (x'',x_{d-1}), x''\in \R^{d-2}$, we see that $|\nabla^2 \phi_{x_0}(0)x'| \le C|x''|$.
	Thus, in view of $\nabla \phi_{x_0}(x') = \nabla^2 \phi_{x_0}(0) x' + O(|x'|^2)$, we have
	\begin{equation*}
		\begin{aligned}
			\{x'\in B(0,r_0): |\nabla \phi_{x_0}(x')| < t \} &\supset \{x'\in B(0,r_0): |\nabla^2\phi_{x_0}(0) x'|  + C|x'|^2 < t\} \\
			& \supset \{x'\in B(0,r_0): C|x''|  + C|x'|^2 < t\} \\
			& \supset \{x'\in B(0,r_0): C|x''| < t, C|x_{d-1}|^2 < t \}
		\end{aligned}
	\end{equation*}
	for sufficiently small $t$. It follows that,
	\begin{equation*}
		\sigma \{x'\in B(0,r_0): |\nabla \phi_{x_0}(x')| < t \} \ge C^{-1} t^{d-2+1/2}
	\end{equation*}
	for some constant $C$ and for sufficiently small $t$. This contradicts to (iii) of Proposition \ref{prop_weakp} with $p = d-1$ and hence the proof is complete.
\end{proof}

\begin{remark}
	The statement (i) is stronger than $(\varkappa\circ n)^{-1} \in L^{p,\infty}(S,d\sigma)$ since it requires rotation invariance for the hypersurfaces, which actually rules out any hypersurfaces containing a non-trivial portion of an affine hypersurface. This is even clear if we notice that the statement (iii), a type of sublevel set estimate, alleges that the hypersurface must bend at a certain degree at each point. However, it is possible to have a flat portion on $S$ so that $(\varkappa\circ n)^{-1}$ is in $L^{p,\infty}(S,d\sigma)$. In fact, if $D$ is the union of all the flat portions of $S$ whose normals satisfy the Diophantine condition, then it suffices to consider the remaining portion $S\setminus D$. Fortunately, with slight modification in (iii), Proposition \ref{prop_weakp} still holds with $S$ replaced by $S\setminus D$.
\end{remark}

\subsection{Hypersurfaces of finite type}
In what follows, we will consider the hypersurfaces of finite type. Note that Definition \ref{def_Omg} is given in a pure analytical way. To see that this condition is relatively mild, we state an equivalent geometrical definition: a compact hypersurface $S$ is of finite type if at least one of the principle curvatures of $S$ does not vanish to infinite order, uniformly at each point. In view of this, we remark that, except for $d-1$ dimensional linear submanifolds (i.e., a portion of an affine hypersurface), some typical cases are allowed for hypersurfaces of finite type, including disconnected hypersurfaces (related to multiply connected domains), saddle points or saddle surfaces and lower dimensional linear submanifolds (such as the side surface of a 3 dimensional cylinder). Also, we mention that any compact real-analytic hypersurface not lying in any affine hypersurface must be of finite type; see \cite{Stein93}.

The next theorem, in connection with the well-known van de Corput's lemma and sublevel set estimate, indicates the importance of the notion of finite type for our application.
\begin{theorem}\label{thm_VanSublevel}
	Let $\phi$ be a smooth function of type $k$ in some sphere $B \subset \R^d$, then
	
	(i) (van de Corput) for any $\psi \in C_0^\infty(B)$,
	\begin{equation}\label{est_van}
		\bigg| \int_{\R^d} e^{i\lambda \phi(x)} \psi(x) dx \bigg| \le C \lambda^{-1/k} \norm{\nabla \psi}_{L^1(B)},
	\end{equation}
	where the constant depends only on $d,k,\delta$ and $\norm{\phi}_{C^{k+1}(B)}$.
	
	(ii) (Sublevel set estimate) for any $t>0$,
	\begin{equation}\label{est_sublevel}
		\sigma \{ x\in (1/2)B: |\phi(x)| \le t \} \le C t^{1/k},
	\end{equation}
	where the constant depends only on $d,k,\delta$ and $\norm{\phi}_{C^{k+1}(B)}$.
\end{theorem}

The proof of part (i) can be found in, e.g., \cite[VIII.2]{Stein93}. The proof of part (ii) follows from part (i) via a simple trick (see \cite[pp. 983]{CCW99}): by writing $u(\phi(x)) = \int_{\R} e^{i\lambda \phi(x)} \hat{u}(\lambda) d\lambda$, decay estimate (\ref{est_van}) translate directly to the estimate on $\int u(\phi(x)) f(x) dx$. In particular, choose non-negative function $f\in C_0^\infty(B)$ so that $f = 1$ on $(1/2)B$ and $u(t) = \chi_{[-1,1]}(t/\alpha)$, where $\chi_{[-1,1]}$ is the characteristic function of $[-1,1]$. Then one sees that estimate (\ref{est_van}) implies (\ref{est_sublevel}).

The following surprising results indicates the equivalence between $(\varkappa \circ On(\cdot))^{-1} \in L^{p,\infty}(S,d\sigma)$ and the finite type condition.

\begin{proposition}\label{prop_typek}
	A closed smooth hypersurface $S\subset \R^d$ is of finite type if and only if the statements of Proposition \ref{prop_weakp} hold with some $p>0$. More precisely, we have:
	
	(i) If $S$ is of type $k$, then the statements of Proposition \ref{prop_weakp} hold with $p = 1/(k-1)$;
	
	(ii) Conversely, if Proposition \ref{prop_weakp} holds with some $p>(d-1)/k$, then $S$ is of type $k$;
	
	(iii) In particular, if $d=2$, then $S$ is of type $k$ if and only if Proposition \ref{prop_weakp} holds with $p = 1/(k-1)$.
\end{proposition}

\begin{proof}
	(i) Let $S$ be of type $k$. By the definition, for any $x_0\in S$, the local graph $x_d = \phi_{x_0}(x')$ is a function of type $k$ in $B(0,r)$ for some $r>0$ with lower bound $\delta>0$. Since $S$ is compact and smooth, the parameters $r$ and $\delta$ involved above can be chosen uniformly in $x_0$. Let $r_0$ and $\delta_0$ be the universal parameters for $S$. It follows that there is some $1\le j\le d-1$ such that $\frac{\partial}{\partial x_j} \phi_{x_0}$ is type $k-1$. By the sublevel set estimate (\ref{est_sublevel}), there exists some constant $C$ independent of $t$ and $x_0$ such that
	\begin{equation*}
		\sigma \Big\{x'\in B(0,r_0): \Big|\frac{\partial}{\partial x_j}\phi_{x_0} (x')\Big | \le t \Big\} \le Ct^{1/(k-1)}.
	\end{equation*}
	Since $\{x'\in B(0,r_0): |\nabla \phi_{x_0}(x')| \le t  \} \subset \{x'\in B(0,r_0): |\frac{\partial}{\partial x_j} \phi_{x_0}(x')| \le t  \}$, one concludes that
	\begin{equation*}
		\sigma \{x'\in B(0,r_0): |\nabla \phi_{x_0}(x')| \le t  \} \le C t^{1/(k-1)}.
	\end{equation*}
	This proves (iii) of Proposition \ref{prop_weakp} and therefore implies all the statements in Proposition \ref{prop_weakp} are true due to the equivalence.
	
	(ii) If $S$ is not type $k$, then by Definition \ref{def_Omg} and the compactness of $S$, there must be some point $x_0\in S$ such that $\partial^\alpha \phi_{x_0}(0) = 0$ for all $|\alpha| \le k$. It follows that $|\nabla \phi_{x_0}(x')| = O(|x'|^k)$ and thereby
	\begin{equation*}
		\sigma\{ x'\in B(0,r_0): |\nabla \phi_{x_0}(x')| \le t \} \ge \sigma\{ x'\in B(0,r_0): C|x'|^k \le t \} = Ct^{(d-1)/k}.
	\end{equation*}
	Obviously, this contradicts to the assumption that Proposition \ref{prop_weakp} holds with some $p > (d-1)/k$.
	
	(iii) Finally, if $d=2$, combining (i) and (ii), we obtain (iii).
\end{proof}

We remark that for $d\ge 3$, there is a gap between the exponents of (i) and (ii) in Proposition \ref{prop_typek}, which arises naturally since, in the worst case, the finite type condition is only satisfied in a certain direction along the tangent plane. So to fill this gap by using the sublevel set estimate, a condition of (strong) finite type applied to all directions and an extra assumption of convexity of certain type on the hypersurfaces may be required; see, e.g., \cite{BNW88,Ruz09}

Next, we will apply the van de Corput's estimate (\ref{est_van}) to prove a homogenization theorem for operators with \emph{constant coefficients} in a domain of finite type. We establish the convergence rate which is optimal in the sense of (\ref{est_van}) and the purpose of doing so is to provide a comparison with later result dealing with oscillating coefficients. Precisely, we consider the following Dirichlet problem with constant coefficients
\begin{equation}\label{eq_Dconst}
	\left\{
	\begin{aligned}
		-\nabla\cdot (A_0\nabla u_\e)(x) &= 0 \quad & \txt{ in }& \Omega, \\
		u_\e(x) &= f(x,x/\e) \quad & \txt{ on } &\partial\Omega.
	\end{aligned}
	\right.
\end{equation}
where constant matrix $A_0$ satisfies (\ref{cdn_ellipticity}) and $f(x,y)$ satisfies (\ref{cdn_periodicity}) and (\ref{cdn_smooth}). Then we have
\begin{theorem}\label{thm_const}
	Let $\Omega$ be a bounded, smooth domain of type $k$. Then the solutions of system (\ref{eq_Dconst}) converges strongly in $L^2(\Omega)$, as $\e\to 0$, to some function $u_0$, which is the solution of
	\begin{equation}\label{eq_Dconst_Homo}
		\left\{
		\begin{aligned}
			-\nabla\cdot (A_0\nabla u_0)(x) &= 0 \quad & \txt{ in }& \Omega, \\
			u_0(x) &= \bar{f}(x) \quad & \txt{ on } &\partial\Omega.
		\end{aligned}
		\right.
	\end{equation}
	where $\bar{f}(x) = \int_{\T^d} f(x,y)dy$.
	Moreover, there exists $C$ independent of $\e$ such that
	\begin{equation*}
		\norm{u_\e - u_0}_{L^2(\Omega)} \le C \e^{1/(2k)}.
	\end{equation*}
\end{theorem}

\begin{proof}
	Write $f_\e(y) = f(y,y/\e)$. Let $P(x,y)$ be the Poisson kernel of operator $ - \nabla\cdot A_0\nabla$ in $\Omega$. By (\ref{eq_Dconst}), (\ref{eq_Dconst_Homo}) and the Poisson integral formula, one has
	\begin{equation*}
		u_\e(x) - u_0(x) = \int_{\partial\Omega} P(x,y) ( f_\e(y) - \bar{f}(y) ) d\sigma(y).
	\end{equation*}
	Now we can localize the integral by applying a partition of unity on $\partial\Omega$. Precisely, we can construct finite smooth functions $\{\eta_i : 1\le i\le N\}$ such that $1 = \sum \eta_i$ on $\partial\Omega$ and $\txt{supp}(\eta_i) \subset B(y_i,r_0), y_i\in \partial\Omega$, where $r_0$ is chosen suitably small. Moreover, $|\nabla \eta_i| \le C$. Therefore,
	\begin{equation}\label{eq_ueu0_eta}
		u_\e(x) - u_0(x) = \sum_{i=1}^{N}\int_{\partial\Omega} P(x,y) ( f_\e(y) - \bar{f}(y) ) \eta_{i}(y) d\sigma(y).
	\end{equation}
	
	Now we fix some $\eta = \eta_i$ with $y_0 = y_i$ and $\txt{supp}(\eta) \subset B(y_0,r_0)$, and consider the integral in (\ref{eq_ueu0_eta}) with $\eta$ involved. By translation and rotation we can transform the surface integral to the usual one in $\R^{d-1}$. Precisely, we assume that $z = O^t(y-y_0)$ moves $y_0 \in \partial\Omega$ to origin and transforms the tangent plane at $y_0$ to $z_d = 0$, where $O$ is an orthogonal matrix. As a result, $\partial\Omega\cap B(y_0,r_0)$ is transformed to the local graph $z_d = \phi(z') = \phi_{y_0}(z')$ which satisfies $\phi(0) = 0$ and $\nabla \phi(0) = 0$. Thus it is sufficient to estimate
	\begin{align}\label{eq_local_Poisson}
		\begin{aligned}
			&\int_{\partial\Omega \cap B(x_0,r_0)} P(x,y)( f_\e(y) - \bar{f}(y) ) \eta(y) dy \\
			& \quad = \int_{\{|z'|<r_0, z_d = \phi(z')\}} P(x,Oz+y_0)( f_\e - \bar{f} )(Oz+y_0) \eta(Oz+y_0) d\sigma(z) \\
			& \quad =  \int_{B(0,r_0)} P(x,Oz+y_0) ( f_\e - \bar{f})(Oz+y_0) \eta(Oz+y_0) \sqrt{1+|\nabla \phi(z')|^2} dz',
	\end{aligned}\end{align}
	where $z = (z',\phi(z'))$.
	

	Next we expand $f_\e(y) - \bar{f}(y)$ in Fourier series, i.e., $f_\e(y) - \bar{f}(y) = \sum f_m(y) e^{i\e^{-1} m\cdot y}$, where the sum is taken over all $m\in \Z^d,m\neq 0$. The last integral now is reduced to the estimate of
	\begin{equation}\label{eq_localm}
		e^{i\e^{-1}m\cdot y_0} \int_{B(0,r_0)} P(x,Oz+y_0)f_m(Oz+y_0) e^{i\e^{-1}O^tm \cdot z} \eta(Oz+y_0) \sqrt{1+|\nabla \phi(z')|^2} dz'.
	\end{equation}
	
	To simplify the expression, let $n = |Q^t m|^{-1} O^tm = (n',n_d)$ and $\varphi(z') = n\cdot z = n'\cdot z' + n_d \phi(z')$. Also, let $\lambda = \e^{-1} |m|^{-1}$ and
	\begin{equation*}
		g_{m}(z) = P(x,Oz+y_0)f_m(Oz+y_0) \eta(Oz+y_0) \sqrt{1+|\nabla \phi(z')|^2}.
	\end{equation*}
	Then, (\ref{eq_localm}) becomes
	\begin{equation}\label{eq_gm}
		e^{i\e^{-1}m\cdot y_0} \int_{B(0,r_0)} g_m(z') e^{i\lambda \varphi(z')} dz'.
	\end{equation}
	In view of the form of $\varphi(z')$, the estimate of oscillatory integral (\ref{eq_gm}) can be divided into two cases.
	
	{\bf Case 1:} $|n_d| < \delta_0$ for some sufficiently small $\delta_0$, say, $\delta_0 < (1/10) \min\{1, \norm{\nabla\phi}_{L^\infty(B(0,r_0))}^{-1} \}$. In this case, it is easy to see that $|\nabla \varphi(z')| \ge 1/2$ on $B(0,1/2)$. By a standard estimate of oscillatory integral, we have
	\begin{equation*}
		\bigg| \int_{B(0,r_0)} g_m(z') e^{i\lambda \varphi(z')} dz' \bigg| \le C \lambda^{-1} \int_{B(0,r_0)} |\nabla g_m|.
	\end{equation*}
	Note that $|\nabla g_m| \le C \norm{f_m}_{C^1} \sum_{k=0,1} |\nabla^k P(x,Oz+y_0)| \le C \norm{f_m}_{C^1} |x-(Oz+y_0)|^{d}$. It follow that
	\begin{equation*}
		\int_{B(0,r_0)} |\nabla g_m| \le C \norm{f_m}_{C^1(\partial\Omega)} \txt{dist}(x,\partial\Omega)^{-1}.
	\end{equation*}
	Hence
	\begin{equation*}
		\bigg| \int_{B(0,r_0)} g_m(z') e^{i\lambda \varphi(z')} dz' \bigg| \le C \e |m| \norm{f_m}_{C^1(\partial\Omega)} \txt{dist}(x,\partial\Omega)^{-1}.
	\end{equation*}
	
	{\bf Case 2:} $n_d \ge \delta_0$, where $\delta_0$ is the same as case 1. Now one takes advantage of the assumption that $\partial\Omega$ is of type $k$ and use Theorem \ref{thm_VanSublevel} (i) to obtain
	\begin{align*}
		\bigg| \int_{B(0,r_0)} g_m(z') e^{i\lambda \varphi(z')} dz' \bigg| & \le C\lambda^{-1/k} \int_{B(0,r_0)} |\nabla g_m| \\
		& \le C \e^{1/k} |m|^{1/k} \norm{f_m}_{C^1(\partial\Omega)} \txt{dist}(x,\partial\Omega)^{-1}.
	\end{align*}
	
	Combining Case 1 and Case 2, and summing $m$ over all $m\in \Z^d\setminus\{0\}$, one obtains the estimate for (\ref{eq_local_Poisson}),
	\begin{equation*}
		\bigg| \int_{\partial\Omega \cap B(x_0,r_0)} P(x,y)( f_\e(y) - \bar{f}(y) ) \eta(y) dy \bigg| \le C \frac{\e^{1/k}}{\txt{dist}(x,\partial\Omega)} \sum_{m\in \Z^d\setminus\{ 0\}} |m|\norm{f_m}_{C^1(\partial\Omega)}.
	\end{equation*}
	Using the smoothness of $f$, one can easily verifies that
	\begin{align*}
		\sum_{m\in \Z^d\setminus\{ 0\}} |m|\norm{f_m}_{C^1(\partial\Omega)} & \le \beta \sum_{m\in \Z^d\setminus\{ 0\}}  |m|^{-d-1} + \beta^{-1} \sum_{m\in \Z^d\setminus\{ 0\}} |m|^{d+3} \norm{f_m}^2_{C^1(\partial\Omega)} \\
		& \le C\sup_{x\in \partial\Omega} \big(\norm{f(x,\cdot)}_{H^{(d+3)/2}(\T^d)} + \norm{\nabla_x f(x,\cdot)}_{H^{(d+3)/2}(\T^d)} \big),
	\end{align*}
	where $\beta$ is selected to minimize the second inequality.
	As a consequence, we obtain
	\begin{equation*}
		|u_\e(x) - u_0(x)| \le   \frac{C \e^{1/k}}{\txt{dist}(x,\partial\Omega)}, \qquad \forall x\in \Omega
	\end{equation*}
	
	On the other hand, the Agmon-type maximal principle implies $|u_\e(x) - u_0(x)| \le C$ for all $x\in \Omega$. Hence,
	\begin{align*}
		\int_{\Omega} |u_\e - u_0|^2 & = \int_{\{\txt{dist}(x,\partial\Omega) > \e^{1/k} \}} |u_\e - u_0|^2 + \int_{\{\txt{dist}(x,\partial\Omega) \le \e^{1/k} \}} |u_\e - u_0|^2 \\
		& \le C\e^{2/k} \int_{\{\txt{dist}(x,\partial\Omega) > \e^{1/k} \}} \txt{dist}(x,\partial\Omega)^{-2} dx + C|\{\txt{dist}(x,\partial\Omega) \le \e^{1/k} \}| \\
		& \le C\e^{1/k}.
	\end{align*}
	This ends the proof.
\end{proof}

We end this section by a diagram illustrating the relationships between finite type condition, oscillatory integrals, the property of Diophantine function and homogenization; see Figure \ref{fig_1}. Note that the arrows represent implications. Besides Theorem \ref{thm_main}, all the implications listed in the diagram have been proved or interpreted in this section. It is of independent interest to observe the close relationship between the Diophantine function and oscillatory integrals while further developments regarding this would be interesting as well.

\begin{figure}[H]
	\centering
	\begin{tikzpicture}
	\matrix (m) [matrix of math nodes,row sep=4em,column sep=5em,minimum width=2em]
	{
		& \textrm{van de Corput's estimate} & \\
		\textrm{Finite type condition} & \textrm{Sublevel set estimate} & \textrm{Homogenization}\\ 
		& \begin{tabular}{@{}c@{}}Diophantine function \\ $(\varkappa\circ On)^{-1} \in L^{p,\infty}$\end{tabular} & \\};
	\path[-stealth] [>=stealth]
	
	(m-1-2) edge node [midway, right] {\tiny 
		\begin{tabular}{@{}c@{}} \emph{Thm.} \\ \emph{\ref{thm_VanSublevel} (ii)}\end{tabular}} (m-2-2)
	
	edge node [midway, above, sloped] {\tiny \emph{Constant coeff.}} node [midway, below, sloped] {\tiny \emph{Thm. \ref{thm_const}}}(m-2-3)
	
	(m-2-1) edge node [midway, above, sloped] {\tiny \emph{Thm. \ref{thm_VanSublevel} (i)}} (m-1-2)
	
	(m-2-2.west|-m-2-1) edge node [midway, below] {\tiny \emph{Prop. \ref{prop_typek} }} (m-2-1)
	
	(m-3-2) edge [<->] node [midway, right] {\tiny \emph{Prop. \ref{prop_weakp}}} (m-2-2)
	edge node [midway, below, sloped] {\tiny \emph{Ocillating coeff.}} node [midway,above, sloped] {\tiny \emph{Thm. \ref{thm_main}}} (m-2-3) ;
	\end{tikzpicture}
	\vspace{-1 em}
	\caption{Relationships between properties of hypersurfaces}\label{fig_1}
\end{figure}
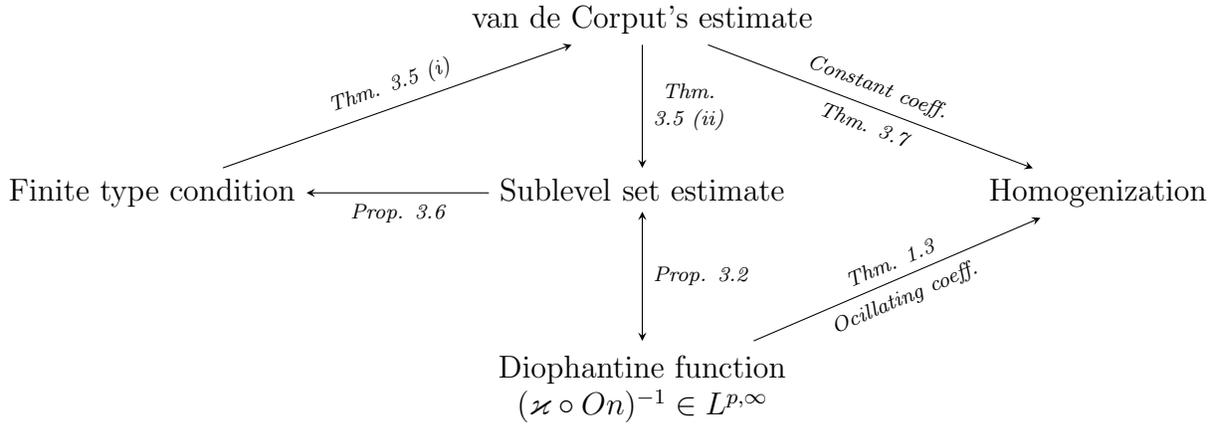

\section{Asymptotic analysis of the Poisson kernel}
\subsection{Auxiliary problems in half-space}
For $n\in \bS^{d-1}$ and $a\in \R$, let $\bH^d_n(a)$ denote the half-space $\{ x\in \R^d: x\cdot n < -a\}$ with $n$ being the unit outer normal to its boundary $\partial \bH^d_n(a) = \{x\in \R^d: x\cdot n = -a \}$.  Consider the Dirichlet problem
\begin{equation}\label{eq_half}
	\left\{
	\begin{aligned}
		- \txt{div} (A\nabla u(x)) &= 0 \qquad & \txt{ in }& \bH^d_n(a), \\
		u(x) &= f(x) \qquad & \txt{ on } &\partial\bH^d_n(a),
	\end{aligned}
	\right.
\end{equation}
where $A$ satisfies (\ref{cdn_ellipticity}), (\ref{cdn_periodicity}) and (\ref{cdn_smooth}), and $f$ is smooth and 1-periodic. Instead of solving (\ref{eq_half}) directly, we try to find a solution of (\ref{eq_half}) with a particular form, i.e., $u(x) = V^a(x-(x\cdot n)n, -x\cdot n)$, where $V^a = V^a(\theta,t)$ is a function of $(\theta,t) \in \T^d\times [a,\infty)$. To identify the system satisfied for $V^a$, let $M$ be a $d\times d$ orthogonal matrix whose last column is $-n$. Let $N$ denote the $d\times (d-1)$ matrix of the first $d-1$ columns of $M$. Since $MM^T = I$, we see that $NN^T + n\otimes n = I$. 
It follows from (\ref{eq_half}) and the previous settings that $V^a$ must be a solution of
\begin{equation}\label{eq_halfV}
	\left\{
	\begin{aligned}
		- \Bigg( \begin{aligned}
			N^T \nabla_\theta \\  \partial_t \ \ \;
		\end{aligned}
		\Bigg) \cdot B
		\Bigg( \begin{aligned}
			N^T \nabla_\theta \\  \partial_t \ \ \;
		\end{aligned}
		\Bigg) V
		&= 0 \qquad & \txt{ in }& \T^d\times (a,\infty), \\
		V &= F \qquad & \txt{ on } & \T^d\times \{a\},
	\end{aligned}
	\right.
\end{equation}
where $B(\theta,t) = M^TA(\theta-tn) M$ and $F(\theta) = f(\theta)$. Observe that if $V^a$ is a solution of (\ref{eq_halfV}) with $a\in \R$, then $ V^a(\theta,t) = V^0(\theta-an,t-a) $, which reduces the problem to the particular case $a = 0$.

Now we collect some important results concerning the lifted system (\ref{eq_halfV}) in the following theorem.

\begin{theorem}\label{thm_halfV}
	Let $n\in \bS^{d-1},a = 0$ and $F\in C^\infty(\T^d)$. Then
	
	(i) The system (\ref{eq_halfV}) has a smooth solution $V$ such that for all $k,s\ge 0$,
	\begin{equation*}
		\int_0^\infty \norm{N^T \nabla_\theta \partial_t^k V}_{H^s(\T^d)}^2 + \norm{\partial_t^{k+1} V}_{H^s(\T^d)}^2 dt \le C,
	\end{equation*}
	where $C$ depends only on $d,m,k,s,A$ and $F$.
	
	(ii) If $n$ satisfies the Diophantine condition with constant $\varkappa>0$ and $V$ is the solution of (\ref{eq_halfV}) given in (i), then there exists a constant $V_\infty$ such that for all $\alpha\in \N^d,k\ge 0$ and $s\ge 0$,
	\begin{equation*}
		|N^T\nabla_\theta \partial_\theta^\alpha \partial_t^k V|+ |\partial_\theta^\alpha \partial_t^{k+1} V| + \varkappa|\partial_\theta^\alpha \partial_t^{k}(V - V_\infty)| \le \frac{C}{(1+\varkappa t)^s},
	\end{equation*}
	where $C$ depends only on $d,m,k,\alpha,s,A$ and $F$.
	
	(iii) Let $n$ satisfy the Diophantine condition with constant $\varkappa>0$ and $\tilde{n}$ be any other unit vector in $\bS^{d-1}$. Let $V$ and $\widetilde{V}$ be the solutions of (\ref{eq_halfV}) corresponding to $n$ and $\tilde{n}$, respectively. Define $W = V - \widetilde{V}$. Then for any $0<\sigma<1$,
	\begin{equation*}
		\int_0^1 \int_{\T^d} |\widetilde{N}^T \nabla_\theta W|^2 + |\partial_t W|^2 d\theta dt \le C \bigg( \frac{|n-\tilde{n}|^4}{\varkappa^{4+\sigma}} + \frac{|n-\tilde{n}|^2}{\varkappa^{2+\sigma}} \bigg).
	\end{equation*}
	where $(\widetilde{N},-\tilde{n})$ is an orthogonal matrix and $C$ depends only on $d,m,\sigma,A$ and $F$.
\end{theorem}

The proofs of (i) and (ii) are more or less well-known and can be found in \cite{GerMas11,AKMP16,Pra13}. Statement (iii) was established in \cite{ShenZhuge16} recently for Neumann problems by applying a weighted estimate. The proof for Dirichlet problems is similar without any real difficulty.

For problems with non-convex domains, We will also need the extension of $V$ from $\T^d\times [a,\infty)$ to the whole space $\T^d\times \R$. We state the result as follows.

\begin{proposition}\label{prop_ExtP}
	If $V \in C^k(\T^d\times [a,\infty))$, then it has an extension $\bar{V} \in C^k(\T^d\times \R)$ such that
	\begin{equation*}
		\bar{V}(\theta,t) = V(\theta,t) \qquad \txt{if }\ t\ge a,
	\end{equation*}
	and
	\begin{equation}\label{est_barV}
		\norm{\bar{V}}_{C^k(\T^d\times [a-r,a+r])} \le C \norm{V}_{C^k(\T^d\times [a,a+(k+1)r])},
	\end{equation}
	where $C$ depends only on $k$ and $d$.
\end{proposition}
\begin{proof}
	This may be proved by a standard construction; see \cite{AF03}, for example. Without loss of generality, we may assume $a = 0$. Then define
	\begin{equation}\label{eq_barV}
		\bar{V}(\theta,t) = \left\{
		\begin{aligned}
			&V(\theta,t) \qquad & \txt{ if }& t\ge 0, \\
			& \sum_{j=1}^{k+1} \lambda_j V(\theta,-jt) \qquad & \txt{ if } & t<0,
		\end{aligned}
		\right.
	\end{equation}
	where $\lambda_1,\lambda_2,\cdots,\lambda_{k+1}$ are the unique solution of the system of $k+1$ linear equations
	\begin{equation}\label{eq_lmdj}
		\sum_{j=1}^{k+1} (-j)^i \lambda_j=1,\qquad i = 0,1,\cdots, k.
	\end{equation}
	Then one may verify that $\bar{V} \in C^k(\T^d\times \R)$ and (\ref{est_barV}) is satisfied.
\end{proof}

As mentioned before, one solution of (\ref{eq_half}) can be given by
\begin{equation}\label{eq_uV}
	u(x) = V^a(x-(x\cdot n)n, -x\cdot n) = V(x-(x\cdot n+a)n, -x\cdot n -a),
\end{equation}
where $V$ is the solution of (\ref{eq_halfV}) with $a = 0$, given by Theorem \ref{thm_halfV}. Due to Proposition \ref{prop_ExtP}, we can extend $u$ from $\bH^d_n(a)$ to $\R^d$. Actually, one can define
\begin{equation}\label{eq_baru}
	\begin{aligned}
		\bar{u}(x) & = \bar{V}(x-(x\cdot n+a)n, -x\cdot n -a) \\
		& =\left\{
		\begin{aligned}
			& V(x-(x\cdot n+a)n, -x\cdot n -a) \qquad & \txt{ if }& x\cdot n \le  -a, \\
			& \sum_{j=1}^{k+1} \lambda_j V(x-(x\cdot n+a)n, j(x\cdot n +a)) \qquad & \txt{ if } & x\cdot n > -a,
		\end{aligned}
		\right.
	\end{aligned}
\end{equation}
where $\lambda_j$ is given by (\ref{eq_lmdj}). Then $\bar{u}$ is a $C^k$ extension of $u$ in $\R^d$. Moreover, there exist constants $C,c >1$ depending at most on $k,d$ and $m$ such that for any $x_0 \in \partial\bH^d_n(a)$ and $r>0$,
\begin{equation}\label{est_baru}
	\norm{\bar{u}}_{C^k(B(x_0,r))} \le C \norm{u}_{C^k(B(x_0,cr)\cap \bH_n^d(a))}.
\end{equation}

\subsection{Local two-scale expansion}
Throughout this subsection we assume that $\Omega$ is a bounded smooth domain in $\R^d,d\ge 2$, and that $A$ satisfies (\ref{cdn_ellipticity}), (\ref{cdn_periodicity}) and (\ref{cdn_smooth}). 
Consider the Dirichlet problem
\begin{equation}\label{eq_uefe}
	\left\{
	\begin{aligned}
		\cL_\e u_\e(x) &= 0 \qquad & \txt{ in }& \Omega, \\
		u_\e(x) &= f_\e(x) = \e f(x/\e) \qquad & \txt{ on } &\partial\Omega,
	\end{aligned}
	\right.
\end{equation}
where $f(y)$ is 1-periodic and smooth.

Fix $x_0\in \partial\Omega$. The main goal of this subsection is to find an approximation of $u_\e$ in a neighborhood of $x_0$. To this end, we solve the Dirichlet problem in a half-space
\begin{equation}\label{eq_vefe}
	\left\{
	\begin{aligned}
		\cL_\e v_\e(x) &= 0 \qquad & \txt{ in }& \bH^d_{n_0}(a), \\
		v_\e(x) &= f_\e(x) \qquad & \txt{ on } &\partial\bH^d_{n_0}(a),
	\end{aligned}
	\right.
\end{equation}
where $a = -x_0\cdot n_0$ and $\partial\bH^d_{n_0}(a)$ is the tangent plane of $\partial\Omega$ at $x_0$. Note that $v_\e$ has a form of $v_\e(x) = \e v_1(x/\e)$, and $v_1$ is the solution of
\begin{equation}\label{eq_v1f1}
	\left\{
	\begin{aligned}
		\cL_1 v_1(x) &= 0 \qquad & \txt{ in }& \bH^d_{n_0}(a/\e), \\
		v_1(x) &= f(x) \qquad & \txt{ on } &\partial\bH^d_{n_0}(a/\e),
	\end{aligned}
	\right.
\end{equation}

The existence of the solution of (\ref{eq_v1f1}) or (\ref{eq_vefe}) as well as its estimates have been established via the half-space problem in Theorem \ref{thm_halfV} (i) and formula (\ref{eq_uV}). Note that $\Omega\setminus \bH^d_{n_0}(a)$ is non-empty for non-convex domain $\Omega$. To approximate $u_\e$ in $\Omega$, we could extend $v_1$ and hence $v_\e$, in the form of (\ref{eq_baru}) with $k = 2$, to the whole space $\R^d$.
Let $\bar{v}_\e$ denote the extended function of $v_\e$. It follows from (\ref{eq_uV}), (\ref{eq_baru}) and (\ref{est_baru}) that
\begin{equation}\label{eq_barve}
	\bar{v}_\e = v_\e \txt{ in } \bH^d_{n_0}(a) \quad \txt{and} \quad \sup_{x\in B_r}|\nabla^k  \bar{v}_\e(x)| \le C\sup_{B_{cr}\cap \bH_{n_0}^d(a)} |\nabla^k v_\e(x)|, \quad k = 0,1,2.
\end{equation}
for any $r>0$, where $B_r$ and $B_{cr}$ are centered at $x_0$.

Define $w_\e(x) = u_\e(x) - \bar{v}_\e(x)$. Observe that by the definition of $\bar{v}_\e$, $w_\e$ is a solution of $\cL_\e w_\e(x) = 0$ only in $\Omega\cap \bH^d_{n_0}(a)$. Now we prove the following.

\begin{theorem}\label{thm_dwe}
	Let $w_\e$ be constructed as above. Let $\e \le r\le \sqrt{\e}$. Then for any $\sigma \in (0,1)$,
	\begin{equation}\label{est_dwe_O}
		\norm{\nabla w_\e}_{L^\infty(B(x_0,r) \cap \Omega)} \le C\sqrt{\e} + C\frac{r^{2+\sigma}}{\e^{1+\sigma}},
	\end{equation}
	where $C$ depends on $d,m,\mu,\sigma,\Omega,A$ and $f$.
\end{theorem}

To prove the theorem, we require the following lemmas.

\begin{lemma}\label{lem_ueCk}
	Let $u_\e$ be a solution of (\ref{eq_uefe}), then one has for any $k\ge 0$,
	\begin{equation}\label{est_dkue}
		\norm{\nabla^k u_\e}_{L^\infty(\Omega)} \le C\e^{1-k},
	\end{equation}
	where $C$ is independent of $\e$.
\end{lemma}
\begin{proof}
	For $k=0$, we use the Agmon-type maximal principle to obtain
	\begin{equation}\label{est_ue_inf}
		\norm{u_\e}_{L^\infty(\Omega)} \le C\norm{f_\e}_{L^\infty(\partial\Omega)} \le C\e.
	\end{equation}
	For $k > 0$, we apply a blow-up argument. Set $u_\e(x) = \e u_1(x/\e)$. Then $u_1$ is a solution of
	\begin{equation}\label{eq_u1f1}
		\left\{
		\begin{aligned}
			\cL_1 u_1(x) &= 0 \qquad & \txt{ in }& \Omega^{\e}, \\
			u_1(x) &= f(x) \qquad & \txt{ on } &\partial\Omega^{\e},
		\end{aligned}
		\right.
	\end{equation}
	where $\Omega^\e = \{x: \e x \in \Omega  \}$. Note that the $C^k$ character of $\Omega^{\e}$ is controlled by that of $\Omega$. It follows from the local Schauder's estimate that for any $x\in \overline{\Omega^{\e}}$,
	\begin{equation*}
		\norm{\nabla^k u_1}_{L^\infty(B(x,1)\cap \Omega^{\e})} \le C\norm{u_1}_{L^\infty(B(x,2)\cap \Omega^{\e})} + \norm{f}_{C^{k,\alpha}(B(x,2)\cap \Omega^{\e})}.
	\end{equation*}
	Since $f$ is 1-periodic, then $\norm{f}_{C^{k,\alpha}(B(x,2)\cap \Omega^{\e})} \le C\norm{f}_{C^{k,\alpha}(\T^d)}$. And by (\ref{est_ue_inf}), $\norm{u_1}_{L^\infty(\Omega^{\e})} \le C$. It follows that
	\begin{equation*}
		\norm{\nabla^k u_1}_{L^\infty(\Omega^{\e})} \le C,
	\end{equation*}
	for any $k> 0$, where $C$ depends also on $k$. Changing variables back to $u_\e$, we obtain the desired estimates (\ref{est_dkue}).
\end{proof}

\begin{lemma}
	Let $\bar{v}_\e$ be constructed as above, then one has for $k=0,1,2$,
	\begin{equation}\label{est_dkve}
		\norm{\nabla^k \bar{v}_\e}_{L^\infty(\R^d)} \le C\e^{1-k},
	\end{equation}
	where $C$ is independent of $\e$.
\end{lemma}
\begin{proof}
	In view of (\ref{eq_barve}), it suffices to consider the estimates for $v_\e$. Let $v_\e(x) = \e v_1(x/\e)$. Then $v_1$ is the solution of (\ref{eq_v1f1}), which can also be given by the Poisson integral formula
	\begin{equation}\label{eq_v1Poi}
		v_1(x) = \int_{\partial \bH_{n_0}^d(a/\e)} P_{\bH}(x,y) f(y) d\sigma(y),
	\end{equation}
	where $P_{\bH}$ is the Poisson kernel of $\cL_1$ in the half-space $ \bH_{n_0}^d(a/\e)$. A similar estimate as (\ref{est_Poi2}) in half-spaces was established in \cite{GerMas12}, i.e.,
	\begin{equation*}
		P_{\bH}(x,y) \le \frac{C \txt{dist}(x,\partial \bH_{n_0}^d(a/\e))}{|x-y|^d}, \qquad \txt{for all } x\in \bH_{n_0}^d(a/\e).
	\end{equation*}
	Then it follows from (\ref{eq_v1Poi}) that  $\norm{v_1}_{L^\infty( \bH_{n_0}^d(a/\e))} \le C\norm{f}_{L^\infty(\partial \bH_{n_0}^d(a/\e))}$ (Agmon-type maximal principle). Thus, $\norm{v_\e}_{L^\infty(\bH_{n_0}^d(a))} \le C \e$ as desired for $k=0$. The estimates for $k>0$ follow similarly as Lemma \ref{lem_ueCk} by the local Schauder's estimates.
\end{proof}

\begin{proof}[Proof of Theorem \ref{thm_dwe}]
	The proof follows a line of \cite{AKMP16}, with modifications made to adjust to our setting of non-convex domains. 
	
	{\bf Step 1:} Set up and conventions. First of all, since $\Omega$ is smooth, then for each point $x_0$ on $\partial\Omega$, there is another domain $\tO $ satisfying the following:
	\begin{enumerate}
		\item $\tO \subset \Omega\cap \bH^d_{n_0}(a)$;
		\item $\tO$ shares the same tangent hyperplane $\partial\bH^d_{n_0}(a)$ of $\Omega$ at point $x_0$;
		\item $\tO$ is a $C^{2,\alpha}$ domain whose $C^{2,\alpha}$ character is controlled by that of $\Omega$.
	\end{enumerate}
	The existence of such domains is obvious since smooth domains always satisfy the uniform interior spheres condition.
	
	Let $y\in \partial\Omega$ and $|y-x_0|\le r_0$ for some $r_0$ depending only on $\Omega$. We will use the following conventions: let $\hat{y}$ denote the projection of $y$ on $\partial \bH^d_{n_0}(a)$ such that $y - \hat{y}$ is a multiple of $n_0$; let $\tilde{y}$ denote the first point on $\partial\tO$ such that $y - \tilde{y}$ is a multiple of $n_0$. Since both $\Omega$ and $\tO$ are at least $C^2$ near $x_0$, it is easy to see that for all $y$ satisfying $|y-x_0| \le r_0$,
	\begin{equation}\label{est_yx0}
		|y - \tilde{y}| + |y - \hat{y}| + |\tilde{y} - \hat{y}| \le C |y - x_0|^2.
	\end{equation}
	This also implies $|y-x_0| \approx |\tilde{y} - x_0| \approx |\hat{y} - x_0|$ if $r_0$ is sufficiently small. On the other hand, let $n(y)$ and $\tilde{n}(\tilde{y})$ denote the unit outer normal of $\partial\Omega$ and $\partial\tO$, respectively. Then
	\begin{equation}\label{est_nx0}
		|n(y) - n_0| + |\tilde{n}(\tilde{y}) - n_0| + |\tilde{n}(\tilde{y}) - n(y)| \le C|y-x_0|.
	\end{equation}
	
	{\bf Step 2:} We prove the estimate (\ref{est_dwe_O}) in $B(x_0,r)\cap \tO$ which is a subset of $B(x_0,r)\cap \Omega$, i.e.,
	\begin{equation}\label{est_dwe_tO}
		\norm{\nabla w_\e}_{L^\infty(B_r \cap \tO)} \le C\sqrt{\e} + C\frac{r^{2+\sigma}}{\e^{1+\sigma}},
	\end{equation}
	where $B_r = B(x_0,r)$ and $\e\le r\le \sqrt{\e}$. The idea for the proof of (\ref{est_dwe_tO}) is similar to the case of convex domains since, by the definition of $w_\e$, $w_\e$ is a solution of
	\begin{equation*}
		\cL_\e w_\e = 0 \quad \txt{subject to certain Dirichlet boundary condition on } \partial\tO.
	\end{equation*}
	Indeed, it follows from the uniform Lipschitz estimate in $C^{1,\alpha}$ domains that
	\begin{align}\label{est_we_Lip}
		\begin{aligned}
			\norm{\nabla w_\e}_{L^\infty(B_r\cap\tO)} & \le Cr^{-1} \norm{w_\e}_{L^\infty(B_{2r}\cap \tO)} \\
			& \qquad + C \norm{\nabla_{\tan} w_\e }_{L^\infty(B_{2r}\cap \partial\tO)} + C r^\sigma\norm{\nabla_{\tan} w_\e}_{C^\sigma(B_{2r}\cap \partial\tO)}.
		\end{aligned}
	\end{align}
	Note that $\nabla_{\tan} $ can be written as $ (I - \tilde{n}\otimes \tilde{n}) \nabla$ (which can be viewed as the projection of $\nabla$ onto the tangent planes $\tilde{n}^\perp$), where $\tilde{n}$ is the unit outer normal of $\partial\tO$.
	
	We now deal with the estimate of $\nabla_{\tan} w_\e$ on $B_{2r}\cap \tO$. Recall that $w_\e = u_\e - v_\e$ in $\tO$ since $\bar{v}_\e = v_\e$ in $\tO$. Using the fact $u_\e = \e f(x/\e)$ on $\partial\Omega$, we know $(I - n\otimes n)\nabla (u_\e - \e f(x/\e))(y) = 0$ on $\partial\Omega$. It follows that
	\begin{align}\label{est_due_fe}
		\begin{aligned}
			& |(I - \tilde{n}\otimes \tilde{n}) \nabla (u_\e - f_\e)(\tilde{y})| \\
			&\qquad = |(I - \tilde{n}\otimes \tilde{n}) \nabla (u_\e - f_\e)(\tilde{y}) - (I - n\otimes n) \nabla (u_\e - f_\e)(y)| \\
			& \qquad \le |\tilde{n}\otimes \tilde{n} - n\otimes n| \norm{\nabla(u_\e -f_\e)}_{L^\infty(B_{2r}\cap \Omega)} + |\tilde{y} - y| \norm{\nabla^2(u_\e -f_\e)}_{L^\infty(B_{2r}\cap \Omega)} \\
			& \qquad \le C|y-x_0| + C \frac{|y-x_0|^2}{\e},
		\end{aligned}
	\end{align}
	where we have used the mean value theorem in the first inequality, and used (\ref{est_dkue}), (\ref{est_yx0}) and (\ref{est_nx0}) in the second one. Similarly, taking advantage of the fact $v_\e  = f_\e$ on the hyperplane $\partial\bH^d_{n_0}(a)$, we have $(I - n_0\otimes n_0)\nabla(v_\e - f_\e)(\hat{y}) = 0$. By the same argument as (\ref{est_due_fe}), we obtain
	\begin{equation}\label{est_dve_fe}
		|(I - \tilde{n}\otimes \tilde{n}) \nabla (v_\e - f_\e)(\tilde{y})| \le C|y-x_0| + C \frac{|y-x_0|^2}{\e}.
	\end{equation}
	Combining (\ref{est_due_fe}) and (\ref{est_dve_fe}), we have
	\begin{equation*}
		\norm{\nabla_{\tan} w_\e}_{L^\infty(B_{2r}\cap \partial\tO)}  = \norm{(I-\tilde{n}\otimes \tilde{n})\nabla(u_\e - v_\e)}_{L^\infty(B_{2r}\cap \partial\tO)} \le Cr + C\frac{r^2}{\e} \le C\frac{r^2}{\e},
	\end{equation*}
	where the last inequality holds for $r\ge \e$.
	
	A similar argument also shows that $\norm{\nabla_{\tan}^2 w_\e}_{L^\infty(B_{2r}\cap \partial\tO)} \le C \e^{-2} r^2$, which, by interpolation, implies $\norm{\nabla_{\tan} w_\e}_{C^{\sigma}(B_{2r}\cap \partial\tO)} \le C\e^{-1-\sigma} r^{2}$ for any $0<\sigma<1$.
	As a result, to see (\ref{est_dwe_tO}), it is left to estimate $\norm{w_\e}_{L^\infty(B_{2r} \cap \tO)}$.

	{\bf Step 3:} To estimate $w_\e(x)$ in $B_{2r}\cap \tO$, we first claim that
	\begin{equation}\label{est_we_ty}
		|w_\e(\tilde{y})| \le C |y-x_0|^2 \qquad \txt{for all } \tilde{y} \in \partial\tO \cap B(x_0,r_0).
	\end{equation}
	Actually, write agian $w_\e = (u_\e - f_\e) - (v_\e - f_\e)$. Using the cancellation $u_\e - f_\e = 0$ on $\partial\Omega$ and mean value theorem, we have 
	\begin{align*}
		|u_\e(\tilde{y}) - f_\e(\tilde{y})|  & = |u_\e(\tilde{y}) - f_\e(\tilde{y}) - (u_\e(y) - f_\e(y))| \\
		& \le C|\tilde{y} - y| \norm{\nabla(u_\e -f_\e)}_{L^\infty(B_{2r}\cap \Omega)} \\
		& \le C|y-x_0|^2,
	\end{align*}
	where in the last inequality we have used (\ref{est_dkue}) and (\ref{est_yx0}). The estimate for $|v_\e(\tilde{y}) - f_\e(\tilde{y})|$ is the same, which proves (\ref{est_we_ty}).
	
	Then we take advantage of the Poisson integral formula and split it into two parts,
	\begin{equation}\label{eq_Poisson}
		\begin{aligned}
			w_\e(x) & = \int_{\partial \tO} P_{\tO,\e}(x,\tilde{y}) w_\e(\tilde{y}) d\sigma(\tilde{y}) \\
			& = \int_{\partial\tO \cap \{|\tilde{y} - x_0|\le c\sqrt{\e}\}} P_{\tO,\e}(x,\tilde{y}) w_\e(\tilde{y}) d\sigma(\tilde{y}) + \int_{\partial\tO \cap \{|\tilde{y} - x_0| > c\sqrt{\e}\}} P_{\tO,\e}(x,\tilde{y}) w_\e(\tilde{y}) d\sigma(\tilde{y})
		\end{aligned}
	\end{equation}
	where $P_{\tO,\e}$ is the Poisson kernel of $\cL_\e$ in $\tO$ and satisfies the same estimate as (\ref{est_Poi2}) with $\Omega$ replaced by $\tO$. 
	
	To estimate the first term on the right-hand side of (\ref{eq_Poisson}), we apply (\ref{est_Poi2}) and (\ref{est_we_ty}),
	\begin{equation*}
		\begin{aligned}
			&\Bigg| \int_{\partial\tO \cap \{|\tilde{y} - x_0|\le c\sqrt{\e}\}} P_{\tO,\e}(x,\tilde{y}) w_\e(\tilde{y}) d\sigma(\tilde{y}) \Bigg| \\
			& \le C  \int_{\partial\tO \cap \{|\tilde{y} - x_0|\le c\sqrt{\e}\}} \txt{dist}(x,\partial\tO)\frac{|y - x_0|^2}{|x-\tilde{y}|^d} d\sigma(\tilde{y}) \\
			& \le C  \int_{\partial\tO \cap \{|\tilde{y} - x_0|\le c\sqrt{\e}\}} \txt{dist}(x,\partial\tO)\frac{|x - x_0|^2}{|x-\tilde{y}|^d} d\sigma(\tilde{y}) + C  \int_{\partial\tO \cap \{|\tilde{y} - x_0|\le c\sqrt{\e}\}} \frac{\txt{dist}(x,\partial\tO)}{|x-\tilde{y}|^{d-2}} d\sigma(\tilde{y}) \\
			& \le C |x-x_0|^2  + C\txt{dist}(x,\partial\tO) \sqrt{\e} \\
			& \le Cr^2 + r\sqrt{\e},
		\end{aligned}
	\end{equation*}
	where we have used the observation $|y-x_0|^2 \le C|\tilde{y} - x_0|^2 \le C|\tilde{y} - x|^2 + C|x-x_0|^2$.
	
	To bound the second term on the right-hand side of (\ref{eq_Poisson}), we note that (\ref{est_dkue}) and (\ref{est_dkve}) give $\norm{w_\e}_{L^\infty(\tO)} \le C\e$. Then
	\begin{equation*}
		\begin{aligned}
			\Bigg| \int_{\partial\tO \cap \{|\tilde{y} - x_0| > c\sqrt{\e}\}} P_{\tO,\e}(x,\tilde{y}) w_\e(\tilde{y}) d\sigma(\tilde{y}) \Bigg| & \le C\e \int_{\partial\tO \cap \{|\tilde{y} - x_0| > c\sqrt{\e}\}} \frac{\txt{dist}(x,\partial\tO)}{|\tilde{y} - x|^d} d\sigma(\tilde{y}) \\
			& \le C\e \txt{dist}(x,\partial\tO) (\sqrt{\e})^{-1} \le Cr\sqrt{\e}.
		\end{aligned}
	\end{equation*}
	It follows
	\begin{equation*}
		|w_\e(x)| \le Cr^2 + Cr\sqrt{\e}, \qquad \txt{for all } x\in B(0,2r)\cap\tO.
	\end{equation*}
	This, together with (\ref{est_we_Lip}) and the estimates for $\nabla_{\tan} w_\e$ in Step 2, proves (\ref{est_dwe_tO}). 
	
	{\bf Step 4:} Finally, to extend estimate (\ref{est_dwe_tO}) to $B_r\cap \Omega$, it suffices to note that $\partial\tO$ and $\partial \Omega$ are very close near $x_0$ and the $C^2$ regularity of $\bar{v}_\e$ are preserved, thanks to \ref{est_dkve}. Actually, for any point $y^* \in B_r \cap \Omega\setminus\tO$, there exist $y\in B_{cr}\cap \partial\Omega$ and corresponding $\tilde{y} \in B_{cr} \cap \partial\tO$ such that $y^*$ is on the segment connecting $y$ and $\tilde{y}$. Then
	\begin{equation*}
		|\nabla w_\e(y^*) - \nabla w_\e(\tilde{y})| \le C\norm{\nabla^2 w_\e}_{L^\infty(B_{2r})} |y^* - \tilde{y}| \le C\frac{|y - \tilde{y}|}{\e} \le C\frac{r^2}{\e},
	\end{equation*}
	where we have used (\ref{est_yx0}) in the last inequality. Finally, applying (\ref{est_dwe_tO}) for $\nabla w_\e(\tilde{y})$, we obtain
	\begin{equation*}
		|\nabla w_\e(y^*) \le C\sqrt{\e} + C\frac{r^{2+\sigma}}{\e^{1+\sigma}}, \qquad \txt{for all } y^* \in  B_r \cap \Omega\setminus\tO.
	\end{equation*}
	Combing this with (\ref{est_dwe_tO}), we obtain (\ref{est_dwe_O}) as desired.
\end{proof}

In view of (\ref{eq_Poisson_exp}), to study the oscillating behavior of $\omega_\e$, the difficulty is to understand the behavior of $ \nabla \Phi_\e^*$ near the boundary. This can be done by applying Theorem \ref{thm_dwe} to $u_{\e,j}^{*\beta} = \Phi^{*\beta}_{\e,j}(x) - P_j^\beta(x) - \e\chi_j^{*\beta}(x/\e)$ for each $1\le j\le d, 1\le \beta\le m$. Clearly, by definitions of $\Phi$ and $\chi$, $u_{\e,j}^\beta$ satisfies
\begin{equation}\label{eq_uechie}
	\left\{
	\begin{aligned}
		\cL_\e^* u_{\e,j}^{*\beta}(x) &= 0 \qquad & \txt{ in }& \Omega, \\
		u_{\e,j}^{*\beta}(x) &= -\e\chi_j^{*\beta}(x/\e) \qquad & \txt{ on } &\partial\Omega.
	\end{aligned}
	\right.
\end{equation}
For each fixed $x_0\in \partial\Omega$, the system (\ref{eq_vefe}) associated with the adjoint operator $\cL_\e^*$ and $f_\e = -\e \chi_j^{*\beta}(x/\e)$ has a solution $v_{\e,j}^{*\beta}$ of form
\begin{equation*}
	v_{\e,j}^{*\beta}(x) = \e V_j^{*\beta} \bigg( \frac{x-(x\cdot n_0 + a)n_0}{\e}, - \frac{x\cdot n_0 + a}{\e} \bigg), \qquad \txt{for } x\cdot n_0 \le -a,
\end{equation*}
where $a = -x_0\cdot n_0$ and $V_j^{*\beta} = V_j^{*\beta}(\theta,t)$ is a solution of
\begin{equation*}
	\left\{
	\begin{aligned}
		- \Bigg( \begin{aligned}
			N^T \nabla_\theta \\  \partial_t \ \ \;
		\end{aligned}
		\Bigg) \cdot B^*
		\Bigg( \begin{aligned}
			N^T \nabla_\theta \\  \partial_t \ \ \;
		\end{aligned}
		\Bigg) V_j^{*\beta}
		&= 0 \qquad & \txt{ in }& \T^d\times (0,\infty), \\
		V_j^{*\beta} &= -\chi_j^{*\beta} \qquad & \txt{ on } & \T^d\times \{0\},
	\end{aligned}
	\right.
\end{equation*}
given by Theorem \ref{thm_halfV}. Note that $V_j^{*\beta}$ also depends on $n_0$. Now let $\bar{v}_{\e,j}^\beta$ be the extension of $v_{\e,j}^{*\beta}$ given by (\ref{eq_baru}) with $k=2$ and a change of variables. Precisely,
\begin{equation}\label{eq_barvej}
	\bar{v}_{\e,j}^{*\beta}(x) = 
	\left\{
	\begin{aligned}
		& \e V_j^{*\beta} \bigg( \frac{x-(x\cdot n_0 + a)n_0}{\e}, - \frac{x\cdot n_0 + a}{\e} \bigg) \qquad & \txt{ if }& x\cdot n_0 \le  -a, \\
		& \sum_{j=1}^3 \e \lambda_j V_j^{*\beta} \bigg( \frac{x-(x\cdot n_0 + a)n_0}{\e},  \frac{j(x\cdot n_0 + a)}{\e} \bigg) \qquad & \txt{ if } & x\cdot n_0 > -a.
	\end{aligned}
	\right.
\end{equation}

Then, one may deduce from Theorem \ref{thm_dwe} that
\begin{theorem}
	Let $\e\le r\le \sqrt{\e}$ and $\sigma \in (0,1)$. Then for any $x\in B(x_0,r) \cap \Omega$,
	\begin{equation}\label{est_PhiExp}
		\bigg| \nabla \bigg( \Phi^{*\beta}_{\e,j}(x) - P_j^\beta(x) - \e\chi_j^{*\beta}(x/\e) - \bar{v}_{\e,j}^{*\beta}(x) \bigg)\bigg|
		\le C\sqrt{\e} + C\frac{r^{2+\sigma}}{\e^{1+\sigma}},
	\end{equation}
	where $C$ depends on $d,m,\mu,\sigma,\Omega,A$ and $f$.
\end{theorem}

\section{A partition of unity}
For simplicity of notation, throughout this section we will write $\varkappa\circ n(x)$ as $\varkappa(x)$ if no ambiguity. For a large class of smooth domains, such as domains of finite type addressed in this paper, it is reasonable to assume that
\begin{equation*}
	\varkappa(\cdot)^{-1} \in L^{p,\infty}(\partial\Omega),
\end{equation*}
for some fixed $p\in (0,d-1]$, where $d\ge 2$. Define
\begin{equation}\label{eq_gamma}
	\gamma = \frac{d-1}{p}.
\end{equation}
Obviously, $\gamma \ge 1$.

In the next lemma, we will construct a Calder\'{o}n-Zygmund-type decomposition adapted to the function $\varkappa(x)$. Essentially this $L^\infty$-based decomposition is a modified version of \cite[Proposition 3.1]{AKMP16} or a special case of $L^q$-based decomposition in \cite[Lemma 7.2]{ShenZhuge16}. However, in our application for general domains, the $L^\infty$ based decomposition is more flexible and convenient. We mention that the partition of unity, provided by the next lemma, will play a crucial role in the analysis of oscillating Dirichlet problems. As in \cite{ShenZhuge16}, we first describe such construction in flat spaces.

\begin{lemma}\label{lem_CZdecomp}
	Let $F$ be a bounded non-negative function on some cube $Q_0\subset \R^{d-1}$. Let $\tau>0$ be a small parameter. Then there exists a finite sequence of dyadic cubes (obtained by bisecting $Q_0$) $\cP = \cP_\tau = \{Q_j:j=1,2,\cdots\}$ such that
	
	(i) The interiors of these cubes are disjoint.
	
	(ii) $Q_0 = \cup_j Q_j$.
	
	(iii) For each $Q_j$,
	\begin{equation}\label{est_Fle}
		\norm{F}_{L^\infty(6Q_j)} \le \frac{\tau}{\ell(Q_j)},
	\end{equation}
	and
	\begin{equation}\label{est_Fge}
		\norm{F}_{L^\infty(6Q^+_j)} > \frac{\tau}{\ell(Q^+_j)},
	\end{equation}
	where $Q_j^+$ is the parent of $Q_j$.
	
	(iv) If $\txt{dist}(Q_j,Q_k) = 0$, then
	\begin{equation*}
		\frac{1}{2} \ell(Q_j) \le |Q_k| \le 2 \ell(Q_j).
	\end{equation*}
	
	(v) There exists an absolute constant $C>0$ such that
	\begin{equation}\label{est_NumOfCubes}
		\#\{Q_j: \ell(Q_j) \ge \lambda \tau \} \le C(\lambda \tau)^{-(d-1)} \sigma(\{ x\in Q_0: F(x) \le \lambda^{-1} \}).
	\end{equation}
\end{lemma}

\begin{proof}[Proof of Lemma \ref{lem_CZdecomp}]
	Let $U$ be the set of all dyadic cubes in $Q_0$ such that $(\ref{est_Fle})$ holds. We say $Q$ is a maximal element of $U$ if $Q$ is not properly contained in any other cube in $U$. Let $\cP$ denote the set of all maximal elements of $U$. Clearly, by definition, the interiors of cubes in $\cP$ are disjoint and (\ref{est_Fle}), (\ref{est_Fge}) are satisfied for each $Q_j$ in $\cP$.
	
	To see (ii), it suffices to note that $F$ is bounded and hence (\ref{est_Fle}) must be satisfied for sufficiently small cubes. And (iv) follows by a similar argument as \cite[Lemma 3.2]{Shen98} which we will omit here.
	
	Finally, to see (v), let $t = 2^{-k} \ell(Q_0)$ be fixed. We consider the number of cubes $Q_j\in \cP$ with $\ell(Q_j) = t$. It follows from (\ref{est_Fle}) that
	\begin{align*}
		\#\{Q_j: \ell(Q_j) = t \} & = t^{1-d} \sigma\Big(\bigcup_{\ell(Q_j) = t} Q_j\Big) \\
		& \le t^{1-d} \sigma( \{ x\in \partial\Omega: F(x) \le t^{-1}\tau \}).
	\end{align*}
	Set $\lambda = t \tau^{-1}$, we obtain
	\begin{equation}\label{est_NoQj}
		\#\{Q_j: \ell(Q_j) = \lambda \tau \} \le (\lambda \tau)^{-(d-1)} \sigma( \{ x\in \partial\Omega: F(x) \le \lambda^{-1} \}).
	\end{equation}
	Finally, replacing $\lambda$ with $2^{k} \lambda$ in the last inequality and summing up all $k\ge 0$, we obtain the desired estimate (\ref{est_NumOfCubes}).
\end{proof}

Let $F\in L^\infty(\partial\Omega)$. Following the lines of \cite{ShenZhuge16}, we can perform a partition of unity on $\partial\Omega$ based on Lemma \ref{lem_CZdecomp}. For readers' convenience, we will provide the outline of the construction.

Fix $x_0\in \partial\Omega$. Let $r_0$ be sufficiently small so that $B(x_0,r_0)\cap \partial\Omega$ is given by the local graph in a coordinate system. Let $\partial\bH_{n_0}^d(a)$ denote the tangent plane for $\partial\Omega$ at $x_0$, where $n_0 = n(x_0)$ and $a = - x_0\cdot n_0$. For $x\in B(x_0,r_0)\cap \partial\Omega$, let
\begin{equation*}
	P(x) = x - ((x-x_0)\cdot n_0) n_0
\end{equation*}
denote its projection on $\partial\bH_{n_0}^d(a)$. Note that $P$ is one-to-one from $B(x_0,r_0)\cap \partial\Omega$ to its image in $\partial\bH_{n_0}^d(a)$ and keeps length and measure comparable. To construct a partition of unity on $B(x_0,r_0)\cap \partial\Omega$ for $F$, we use the inverse map $P^{-1}$ to lift a partition on the tangent plane, given in Lemma \ref{lem_CZdecomp}, to $\partial\Omega$. Precisely, for a fixed cube $Q_0$ in $\partial\bH_{n_0}^d(a)$ such that $B(x_0,2r_0)\cap \partial\Omega \subset P^{-1}(Q_0) \subset B(x_0,4r_0\sqrt{d})\cap \partial\Omega $, we apply Lemma \ref{lem_CZdecomp} to $Q_0$ for function $F\circ P^{-1}$. This generates a finite sequence of dyadic cubes $\{Q_j\}$ satisfying the properties in Lemma \ref{lem_CZdecomp}. Let $x_j$ be the center of $Q_j$ and $r_j$ be the side length. Let $\widetilde{Q}_j = P^{-1}(Q_j)$. Then
\begin{equation*}
	\widetilde{Q}_0 = P^{-1}(Q_0) = \bigcup_j \widetilde{Q}_j
\end{equation*}
gives a decomposition of $\widetilde{Q}_0$. Also, let $\tilde{x}_j = P^{-1}(x_j)$ and $t\widetilde{Q}_j = P^{-1}(tQ_j)$. Now for each $\widetilde{Q}_j$, we choose $\eta_j \in C_0^\infty(\R^d)$ such that $0\le \eta_j \le 1, \eta_j = 1$ on $\widetilde{Q}_j, \eta_j = 0$ on $\partial\Omega\setminus 2\widetilde{Q}_j$, and $|\nabla^k \eta_j| \le Cr_j^{-k}$. Note that by Lemma \ref{lem_CZdecomp} (iv), $1\le \sum_j \eta_j \le C_0$ on $\widetilde{Q}_0$, where $C_0$ is a constant depending only on $d$ and $\Omega$. Finally, we set
\begin{equation*}
	\varphi_j(x) = \frac{\eta_j(x)}{\sum_k \eta_k(x)}.
\end{equation*}
Clearly, $\sum_j \varphi_j = 1$ on $\widetilde{Q}_0$, $0\le \varphi_j \le 1, \varphi_j \ge C_0^{-1}$ on $\widetilde{Q}_j, \varphi_j = 0$ on $\partial\Omega\setminus 2\widetilde{Q}_j$, and $|\nabla^k \varphi_j| \le Cr_j^{-k}$. Further more, the properties (iii) and (v) in Lemma \ref{lem_CZdecomp} are preserved, i.e.,
\begin{equation}\label{est_Flege}
	\norm{F}_{L^\infty(6\widetilde{Q}_j)} \le \frac{\tau}{r_j}, \qquad \norm{F}_{L^\infty(18\widetilde{Q}_j)} > \frac{\tau}{r_j},
\end{equation}
and
\begin{equation}\label{est_NumQ}
	\#\{\widetilde{Q}_j: r_j \ge \lambda \tau \} \le C(\lambda \tau)^{-(d-1)} \sigma(\{ x\in \widetilde{Q}_0: F(x) \le \lambda^{-1} \}).
\end{equation}

For our application in homogenization, we will apply the above decomposition to $F = \varkappa^{1/\gamma}$, where $\gamma$ is defined in (\ref{eq_gamma}). Let $\{\widetilde{Q}_j: j=1,2,\cdots\}$ be the generated cubes on $\partial\Omega$ and other notations are also kept as before. By (\ref{est_Flege}), for each $\widetilde{Q}_j$, there exists $z_j \in 18\widetilde{Q}_j$ such that
\begin{equation}\label{cdn_zj}
	\varkappa(z_j) > \Big( \frac{\tau}{r_j} \Big)^\gamma.
\end{equation}

\begin{lemma}
	There exists some $C>0$ such that for each $j$,
	\begin{equation*}
		\tau \le r_j \le C \sqrt{\tau}.
	\end{equation*}
\end{lemma}
\begin{proof}
	Note that $\varkappa \le 1$. This implies that $r_j \ge \tau$, due to (\ref{est_Flege}). To see the other direction, we first claim that
	for any $0<q<p$, $x\in \partial\Omega$ and $0<r<\txt{diam}(\Omega)$,
	\begin{equation}\label{est_chiq}
		\bigg( \fint_{B(x,r)\cap \partial\Omega} \varkappa^{-q} \bigg)^{1/q} \le \frac{C}{r^\gamma},
	\end{equation}
	where $C$ depends only on $d,q,\gamma$ and $\Omega$. This claim is an extension of \cite[Proposition 7.1]{ShenZhuge16}, whose proof is almost the same. We omit the details here.
	
	Now by (\ref{est_chiq}) and (\ref{est_Flege}),  we have
	\begin{equation*}
		1 \le \bigg( \fint_{6\widetilde{Q}_j} \varkappa^{-q} \bigg)^{1/q} \norm{\varkappa}_{L^\infty(6\widetilde{Q}_j)} \le C r_j^{-\gamma} \Big( \frac{\tau}{r_j} \Big)^{\gamma},
	\end{equation*}
	which implies $r_j \le C\sqrt{\tau}$.
\end{proof}

The following lemma is the same as \cite[Proposition 7.4]{ShenZhuge16} which will be useful to us. We give a simpler proof here based on (\ref{est_NumQ}).
\begin{lemma}\label{lem_rj}
	Let $0<\alpha<d-1$, then
	\begin{equation*}
		\sum_{j} r_j^{d-1+\alpha} \le C\tau^\alpha,
	\end{equation*}
	where $C$ depends only on $\alpha, \gamma, d$ and $\Omega$.
\end{lemma}
\begin{proof}
	It follows from (\ref{est_NumQ}) that
	\begin{align*}
		\sum_j r_j^{d-1+\alpha} & \le \sum_k \sum_{2^{k-1}\tau \le r_j < 2^k \tau} (2^k \tau)^{d-1+\alpha} \\
		& \le C \int_0^\infty (\lambda \tau)^{\alpha} \lambda^{-1} \sigma\{ x\in \widetilde{Q}_0: \varkappa^{-1/\gamma}(x) > \lambda \} d\lambda \\
		& = C \tau^{\alpha} \int_{\partial\Omega} \varkappa^{-\alpha/\gamma} d\sigma \\
		& \le C \tau^{\alpha},
	\end{align*}
	for any $\alpha<d-1$.
\end{proof}

\section{Proof of main theorem}
We first prove Theorem \ref{thm_maina} and then Theorem \ref{thm_main} follows readily from Proposition \ref{prop_typek}. The line of argument is similar to \cite{ShenZhuge16,AKMP16}.

Due to Lemma \ref{lem_uetue}, it is sufficient to estimate $\norm{\tilde{u}_\e - u_0}_{L^2(\Omega)}$, where $\tilde{u}_\e$ and $u_0$ are defined by
\begin{equation}\label{eq_uePoi}
	\tilde{u}_\e^\alpha (x) = \int_{\partial \Omega} P_{\Omega}^{\alpha\gamma}(x,y) \omega_\e^{\gamma\beta}(y) f^\beta (y,y/\e)d\sigma(y) 
\end{equation}
and
\begin{equation}\label{eq_u0Poi}
	u_0^\alpha(x) = \int_{\partial\Omega} P_{\Omega}^{\alpha\gamma}(x,y) \bar{f}^\gamma (y) d\sigma(y).
\end{equation}

Now we need to find an explicit expression for the homogenized data $\bar{f}$. Roughly speaking, the homogenized data $\bar{f}$ in (\ref{eq_u0Poi}) should be the weak limit of $\omega_\e(y) f (y/\e)$ as $\e \to 0$. By (\ref{eq_omgExp}) and (\ref{est_PhiExp}), for $y\in B(x_0,r)\cap \partial\Omega$, one has
\begin{equation}
	\begin{aligned}\label{eq_wefe}
		&\omega_\e^{\gamma\beta}(y) f^\beta (y/\e)\\
		& \quad = h^{\gamma\nu }(y) \cdot n_{\ell}(y) \frac{\partial }{\partial y_\ell} [ P_k^{\rho\nu}(y) + \e\chi_k^{*\rho\nu}(y/\e) + \bar{v}^{*\rho\nu,x_0}_{\e,k}(y) ] n_k(y) \cdot a_{ij}^{\rho\beta}(y/\e) n_i(y)n_j(y) f^\beta (y,y/\e)  \\
		& \qquad + \txt{Error terms}.
	\end{aligned}
\end{equation}
Note that $\bar{v}_\e^{*,x_0}(y)$ is given in (\ref{eq_barvej}) which depends also on $x_0$. For a fixed $y \in \partial\Omega$, in view of the quantitative ergodic theorem \cite[Proposition 2.1]{AKMP16}, we know that $\omega_\e(y) f (y/\e)$ converges to its average on the tangent plane  $\bH_{n}^d(a)$ at $y$, where $n = n(y)$. The only unclear term in (\ref{eq_wefe}) is $n\cdot \nabla \bar{v}_{\e,k}^{*\nu,x_0}$. Actually, in view of (\ref{eq_barvej}), for $z \in \bH_{n}^d(a)$, one has
\begin{equation}\label{eq_ndve}
	n\cdot \nabla \bar{v}_{\e,k}^{*\nu,x_0}(z) = n\cdot (1-n\otimes n, -n) \Bigg( \begin{aligned}
		\nabla_\theta \\  \partial_t \;
	\end{aligned}
	\Bigg)  V_k^{*\nu,x_0}\Big( \frac{z}{\e},0 \Big) = -\partial_t V_k^{*\nu,x_0}\Big( \frac{z}{\e},0 \Big).
\end{equation}
Note that $V_k^{*,x_0}(\theta,t)$ is 1-periodic in $\theta$. As a consequence, without justification, we can define the homogenized boundary data as follows:
\begin{equation}\label{eq_barf}
	\begin{aligned}
		&\bar{f}^\gamma (y) \\
		& =  h^{\gamma\nu }(y) \int_{\T^d} [\delta^{\rho\nu} + n(y) \cdot \nabla \chi^{*\rho\nu}(\theta)\cdot n(y) - \partial_t V^{*\rho\nu,y}(\theta,0) \cdot n(y)] n_i(y)n_j(y)a_{ij}^{\rho\beta}(\theta) f^\beta (y,\theta) d\theta
	\end{aligned}
\end{equation}


\begin{remark}
	If the coefficient matrix $A = (a_{ij}^{\alpha\beta})$ is constant (or divergence free), then $\chi^* = 0$ and hence $V^* = 0$ in (\ref{eq_barf}). Also in this case, one has $\widehat{A} = A$. By the definition of $h$, this implies that $h^{\gamma\nu} \delta^{\rho\nu} n_i n_j a_{ij}^{\rho\beta} = \delta^{\gamma\beta}$. As a result, (\ref{eq_barf}) is reduced to
	\begin{equation*}
		\bar{f}(y) =  \int_{\T^d} f(y,\theta) d\theta.
	\end{equation*}
	This exactly coincides with the homogenized boundary data defined in Theorem \ref{thm_const} for Dirichlet problems with constant coefficients.
\end{remark}

\begin{proposition}\label{prop_barf_reg}
	Let $x,y\in \partial\Omega$ and $|x-y|<r_0$. Suppose that $n(x),n(y)$ satisfies the Diophantine condition with constant $\varkappa(x)$ and $\varkappa(y)$ respectively. Let $\bar{f}$ be defined by (\ref{eq_barf}). Then
	
	(i) For any $\sigma\in (0,1)$,
	\begin{equation*}
		|\bar{f}(x) -  \bar{f}(y)| \le C\bigg( \frac{|x-y|^2}{\varkappa^{2+\sigma}} + \frac{|x-y|}{\varkappa^{1+\sigma}}\bigg) \sup_{z\in \T^d} \norm{f(\cdot,z)}_{C^1(\partial\Omega)}.
	\end{equation*}
	where $\varkappa = \varkappa(x) \vee \varkappa(y)$ and $C$ depends only on $d,m,\sigma,\Omega$ and $A$.
	
	(ii) For any $0<q<q^* = (d-1)/(2\gamma -1)$, one has
	\begin{equation}\label{key}
	\bar{f} \in W^{1,q}\cap L^\infty(\partial\Omega).
	\end{equation}
\end{proposition}

Part (i) of the last proposition is taken from \cite[Theorem 6.1]{ShenZhuge16}, which actually holds for Dirichlet problem as well and follows from Theorem \ref{thm_halfV} (iii). The proof of part (ii) is similar as \cite[Theorem 7.5]{ShenZhuge16} with an obvious modification. Note that the convexity is not necessary in the proofs.

The rest of the proof is devoted to estimating $\norm{u_\e - u_0}_{L^2(\Omega)}$. To begin with, we perform a partition of unity on $\partial\Omega$ and restrict ourself on $B(x_0,r_0) \cap \partial\Omega$ for some $x_0$ and $r_0>0$ sufficiently small.  So without any loss of generality, we may assume $
\txt{supp}(f(\cdot,y)) \subset B(x_0,r_0)$ for any $y\in \T^d$. Then we construct another partition of unity on $B(x_0,r_0) \cap \partial\Omega$ adapted to $F = \varkappa^{1/\gamma}$, by the method described in the last section, with
\begin{equation*}
	\tau = \e^{s},	
\end{equation*}
for some constant $s\in [1/2,1]$, which will be properly selected by optimizing several errors. Thus, there exist a finite sequence of $\{ \varphi_j \}$ of $C_0^\infty$ positive functions in $\R^d$ and a finite of sequence of \emph{surface cubes} $\{\widetilde{Q}_j\}$ on $\partial\Omega$, such that $\sum_j \varphi_j = 1$ on $B(x_0,2r_0)\cap \partial\Omega$. Note that $\varphi_j$ is supported in $2\widetilde{Q}_j$ and $|\nabla^k \varphi_j| \le Cr_j^{-k}$, where $r_j$ is the \emph{side length} of $\widetilde{Q}_j$ as before.

Note that $\tilde{x}_j$ is the center of $\widetilde{Q}_j$. Let $\Gamma_\e$ denote a boundary layer
\begin{equation*}
	\Gamma_\e = \Omega \cap \bigg( \bigcup_{j} B(\tilde{x}_j, Cr_j) \bigg)
\end{equation*}
and $D_\e = \Omega \setminus \Gamma_\e$. By Lemma \ref{lem_rj},
\begin{equation*}
	|\Gamma_\e| \le \sum_j |B(\tilde{x}_j,Cr_j)| \le C \sum_j r_j^d \le C\tau = C\e^{s}.
\end{equation*}
Thus for any $q>0$,
\begin{equation}\label{est_Gamma}
	\int_{\Gamma_\e} |u_\e - u_0|^q \le C\e^{s},
\end{equation}
where we have used the boundedness of $u_\e$ and $u_0$.

To deal with the $L^q$ norm of $u_\e - u_0$ on $D_\e$, we introduce a function (see \cite{ShenZhuge16})
\begin{equation}\label{eq_Thetat}
	\Theta_t(x) = \sum_j \frac{r_j^{d-1+t}}{|x - \tilde{x}_j|^{d-1}},
\end{equation}
where $0\le t < d-1$.

\begin{lemma}\label{lem_Thetat}
	Let $\Theta_t(x)$ be defined by (\ref{eq_Thetat}). Then if $q>0$ and $0\le qt < d-1$,
	\begin{equation*}
		\int_{D_\e} (\Theta_t(x))^q dx \le C \tau^{qt}.
	\end{equation*}
\end{lemma}

The original lemma in \cite{ShenZhuge16} was proved for $q\ge 1$. It follows trivially from H\"{o}lder's inequality that the lemma holds also for $0<q< 1$, which will also be useful for us. This lemma will play a key role and be used repeatedly in the following context.

As in \cite{AKMP16,ShenZhuge16}, we split $\tilde{u}_\e - u_0$ into five parts
\begin{equation*}
	\begin{aligned}
		\tilde{u}_\e(x) - u_0(x) & = \int_{\partial \Omega} P_{\Omega}(x,y) \omega_\e(y) f(y,y/\e)d\sigma(y) - \int_{\partial\Omega} P_{\Omega}(x,y) \bar{f}(y) d\sigma(y)\\
		& = I_1 + I_2 + I_3 + I_4 + I_5,
	\end{aligned}
\end{equation*}
where $I_k, 1\le k \le 5,$ will be defined below and handled separately. We point out in advance that estimates for $I_3$ and $I_4$ essentially distinguish from the case of strictly convex domains and need more careful calculations.

Let $\delta>0$ be an arbitrarily small exponent that might differ in each occurrence.

{\bf Estimate of $I_1$:} Let
\begin{equation*}
	\begin{aligned}
		I_1 =& \int_{\partial \Omega} P_{\Omega}^{\alpha\gamma}(x,y) \omega_\e^{\gamma\beta}(y) f^\beta (y,y/\e)d\sigma(y)  \\
		&\qquad  - \sum_j \int_{\partial\Omega}  \varphi_j(y)P_{\Omega}^{\alpha\gamma}(x,y) \tilde{\omega}_{\e}^{\gamma\beta,z_j}(y) f^\beta (y,y/\e)d\sigma(y),
\end{aligned}
\end{equation*}
where
\begin{equation}\label{eq_tilomega}
	\tilde{\omega}_{\e}^{\gamma\beta,z_j}(y) = h^{\gamma\nu }(y) n_{\ell}(y) \frac{\partial }{\partial y_\ell} \Big[P_k^{\rho\nu}(y) + \e\chi_k^{*\rho\nu}(y/\e) + \bar{v}^{*\rho\nu,z_j}_{\e,k}(y)\Big] n_k(y) a_{im}^{\rho\beta}(y/\e) n_i(y)n_m(y),
\end{equation}
and $z_j$'s are specially selected as in (\ref{cdn_zj}). Note that $I_1$ comes from the error terms in (\ref{eq_wefe}), which by (\ref{est_dwe_O}) is bounded by
\begin{equation*}
	C \sum_j \int_{\partial\Omega} \varphi_j(y) |P_{\Omega}(x,y)| \bigg( \sqrt{\e} + \frac{r_j^{2+\sigma}}{\e^{1+\sigma}}\wedge 1 \bigg)  d\sigma(y) = R_1 + R_2,
\end{equation*}
for any $\sigma\in (0,1)$. Observe that
\begin{equation}\label{est_R1}
	R_1 \le C\sqrt{\e} \int_{\partial\Omega} |P_{\Omega}(x,y)| \le C\sqrt{\e}.
\end{equation}
For $R_2$, using $|P_{\Omega}(x,y)| \le C|x-y|^{1-d}$ and $|x-y| \approx |x - \tilde{x}_j|$ for $x\in D_\e, y \in B(\tilde{x}_j,Cr_j)$, we have
\begin{equation}\label{est_R2}
	R_2 = C \sum_j \int_{\partial\Omega} \varphi_j(y) |P_{\Omega}(x,y)| \bigg( \frac{r_j^{2+\sigma}}{\e^{1+\sigma}}\wedge 1 \bigg)  d\sigma(y)\le C  \e^{-1-\sigma} \sum_j \frac{r_j^{2+\sigma+d-1}}{|x - \tilde{x}_j|^{d-1}}
\end{equation}
Now we estimate $R_2$ by Lemma \ref{lem_Thetat} in two separate cases. If $2(2+\sigma)<d-1$, then we apply Lemma \ref{lem_Thetat} directly with $q = 2$ and obtain
\begin{equation}\label{est_R2a}
	\int_{D_\e} |R_2(x)|^2 dx \le C \e^{-2(1+\sigma)} \tau^{2(2+\sigma)} \le C\e^{4(s - \frac{1}{2}) -\delta},
\end{equation}
where we have used $\tau = \e^s$ and chosen $\sigma$ sufficiently small. Otherwise, we choose suitable $q<2$ such that $q(2+\sigma) = d-1-\sigma <d-1$ and then apply Lemma \ref{lem_Thetat}
\begin{equation*}
	\int_{D_\e} |R_2(x)|^q dx \le C \e^{-q(1+\sigma)} \tau^{q(2+\sigma)} \le C\e^{(s - \frac{1}{2})(d-1)-\delta},
\end{equation*}
where again, $\sigma$ is chosen sufficiently small. Clearly, (\ref{est_R2}) also implies $|R_2| \le C$. Thus, a simple interpolation leads to
\begin{equation}\label{est_R2b}
	\int_{D_\e} |R_2(x)|^2 dx \le C\e^{(s - \frac{1}{2} )(d-1)-\delta}.
\end{equation}
Combining (\ref{est_R1}), (\ref{est_R2a}) and (\ref{est_R2b}), we obtain
\begin{equation*}
	\int_{D_\e} |I_1(x)|^2 dx \le C \e^{1 \wedge 4(s - \frac{1}{2}) \wedge (d-1)(s - \frac{1}{2}) -\delta}.
\end{equation*}

{\bf Estimate of $I_2$:} Set
\begin{equation}\label{eq_I2}
	\begin{aligned}
		I_2 &= \sum_j \int_{\partial\Omega}  \varphi_j(y)P_{\Omega}^{\alpha\gamma}(x,y) \tilde{\omega}_{\e}^{\gamma\beta,z_j}(y) f^\beta (y,y/\e)d\sigma(y)\\
		& \qquad - \sum_j \int_{\partial\bH^d_j}  \varphi_j(P_j^{-1}(y))P_{\Omega}^{\alpha\gamma}(x,P_j^{-1}(y)) \tilde{\omega}_{\e}^{\gamma\beta,z_j}(y) f^\beta (z_j,y/\e)d\sigma(y)
	\end{aligned}
\end{equation}
where $\partial\bH^d_j$ denotes the tangent plane for $\partial\Omega$ at $z_j$ and $P^{-1}_j$ is the inverse of the projection map from $B(z_j,Cr_j) \cap\partial\Omega$ to $\partial\bH^d_j$. We clarify that in (\ref{eq_tilomega}), $n(y)$ is the outer normal of $y \in \partial\Omega$. But in the second term of (\ref{eq_I2}), $y$ needs to belong to $\partial\bH^d_j$ and hence we need to update $n(y) = n(z_j)$ for all $y\in \partial\bH^d_j$. This modification leads to some harmless errors bounded by $Cr_j \le Cr^2_j/\e$. Then, for the same reason as the term $T_2$ in \cite{AKMP16} or $I_2$ in \cite{ShenZhuge16}, we are able to bound $I_2$ by
\begin{equation*}
	|I_2| \le C\e^{-1} \sum_j \frac{r_j^{2+d-1}}{|x - \tilde{x}_j|^{d-1}}.
\end{equation*}
Similar as (\ref{est_R2}), we estimate this in two cases and obtain
\begin{equation*}
	\int_{D_\e} |I_2(x)|^2 dx \le  C\e^{4(s - \frac{1}{2} ) \wedge (d-1)(s - \frac{1}{2}) -\delta}.
\end{equation*}

{\bf Estimate of $I_3$:} Set
\begin{equation*}
	\begin{aligned}
		I_3 & = \sum_j \int_{\partial\bH^d_j}  \varphi_j(P_j^{-1}(y))P_{\Omega}^{\alpha\gamma}(x,P_j^{-1}(y)) \tilde{\omega}_{\e}^{\gamma\beta,z_j}(y) f^\beta (z_j,y/\e)d\sigma(y) \\
		& \qquad - \sum_j \int_{\partial\bH^d_j}  \varphi_j(P_j^{-1}(y))P_{\Omega}^{\alpha\gamma}(x,P_j^{-1}(y)) \bar{f}^\gamma(z_j) d\sigma(y),
	\end{aligned}
\end{equation*}
where $\bar{f}$ is defined in (\ref{eq_barf}). To estimate $I_3$, we apply the quantitative ergodic theorem in \cite{AKMP16}. As we have mention in the estimate of $I_2$, the outer normal in the definition of $\tilde{\omega}_{\e}^{\gamma\beta,z_j}(y)$ is constant on $\partial\bH^d_j$ with Diophantine constant $\varkappa(z_j)$, and therefore $\tilde{\omega}_{\e}^{\gamma\beta,z_j}(y)$ is nothing but a slice of some 1-periodic function in $\R^d$ (see (\ref{eq_ndve})). Note that by (\ref{cdn_zj}), $\varkappa(z_j) > (\tau/r_j)^{\gamma}$. Then it follows from \cite[Proposition 2.1]{AKMP16} that for any $N>0$,
\begin{equation*}
	\begin{aligned}
		|I_3| & \le C\sum_j \Big( \frac{\e r_j^\gamma}{\tau^\gamma} \Big)^{N} \int_{2\widetilde{Q}_j} |\nabla^N(\varphi_j(y) P_{\Omega}(x,y))| d\sigma(y) \\
		& \le C\sum_j \Big( \frac{\e r_j^\gamma}{\tau^\gamma} \Big)^{N} \sum_{k=0}^{N} \frac{r_j^{d-1-N+k}}{|x - \tilde{x}_j|^{d-1+k}} \\
		& \le C \e^N \tau^{-\gamma N} \sum_j \frac{r_j^{N(\gamma - 1) +d-1}}{|x - \tilde{x}_j|^{d-1}},
	\end{aligned}
\end{equation*}
where we have used $|\nabla^k \varphi_j| \le Cr_j^{-k}$, $|\nabla^k P_{\Omega}(x,y)| \le C|x-y|^{1-d-k}$ and $r_j \le C|x-\tilde{x}_j| \approx C|x-y|$ for all $x\in D_\e$ and $y\in 2\widetilde{Q}_j$. Now we choose $q\le 2$ and $N\ge 1$ properly so that $qN(\gamma - 1) = d-1-\delta <d-1$ and apply Lemma \ref{lem_Thetat}
\begin{equation*}
	\int_{D_\e} |I_3|^q \le C\e^{qN} \tau^{-q\gamma N} \tau^{qN(\gamma - 1)} = C \e^{(1-s)(d-1)/(\gamma -1) -\delta}.
\end{equation*}
This implies, as before,
\begin{equation*}
	\int_{D_\e} |I_3|^2 \le C \e^{(1-s)(d-1)/(\gamma -1)-\delta}.
\end{equation*}

{\bf Estimate of $I_4$}: Set
\begin{equation*}
	\begin{aligned}
		I_4 & = \sum_j \int_{\partial\bH^d_j}  \varphi_j(P_j^{-1}(y))P_{\Omega}^{\alpha\gamma}(x,P_j^{-1}(y)) \bar{f}^\gamma(z_j) d\sigma(y) \\
		& \qquad - \sum_j \int_{\partial\bH^d_j}  \varphi_j(P_j^{-1}(y))P_{\Omega}^{\alpha\gamma}(x,P_j^{-1}(y)) \bar{f}^\gamma(P_j^{-1}(y)) d\sigma(y).
	\end{aligned}
\end{equation*}
The estimate for $I_4$ essentially relies on the regularity of homogenized data $\bar{f}$. Indeed, by Proposition \ref{prop_barf_reg}
\begin{align*}
	|\bar{f}(z_j) - \bar{f}(P_j^{-1}(y))| & \le C\Bigg( \frac{r_j^2}{\varkappa(z_j)^{2+\sigma}} + \frac{r_j}{\varkappa(z_j)^{1+\sigma}}\Bigg) \\
	& \le C\Bigg( \frac{r_j^{2+\gamma(2+\sigma)}}{\tau^{\sigma(2+\sigma)}} + \frac{r_j^{1+\gamma(1+\sigma)}}{\tau^{\sigma(1+\sigma)}}\Bigg),
\end{align*}
where we also used $|z_j - P_j^{-1}(y)|\le Cr_j$. This leads to a bound for $I_4$
\begin{align*}
	|I_4| \le C \tau^{-\gamma(2+\sigma)} \sum_j \frac{r_j^{2+\gamma(2+\sigma) + d-1}}{|x - x_j|^{d-1}} + C \tau^{-\gamma(1+\sigma)} \sum_j \frac{r_j^{1+\gamma(1+\sigma) + d-1}}{|x - x_j|^{d-1}},
\end{align*}
of which we denote the terms on the right-hand side by $J_1$ and $J_2$ in proper order. Using Lemma \ref{lem_Thetat} and a familiar argument as before, we are able to show
\begin{equation*}
	\begin{aligned}
		\int_{D_\e} |I_4|^2 &\le C \int_{D_\e} |J_1|^2 + C\int_{D_\e} |J_2|^2 \\
		& \le C \e^{4s \wedge s(d-1)/(1+\gamma)-\delta} + C\e^{2s \wedge s(d-1)/(1+\gamma)-\delta} \\
		& \le C\e^{2s \wedge s(d-1)/(1+\gamma)-\delta}.
	\end{aligned}
\end{equation*}

{\bf Estimate of $I_5$}: Finally, let
\begin{equation*}
	\begin{aligned}
		I_5 & =  \sum_j \int_{\partial\bH^d_j}  \varphi_j(P_j^{-1}(y))P_{\Omega}^{\alpha\gamma}(x,P_j^{-1}(y)) \bar{f}^\gamma(P_j^{-1}(y)) d\sigma(y) \\
		&\qquad - \int_{\partial\Omega} P_{\Omega}(x,y) \bar{f}(y) d\sigma(y).
	\end{aligned}
\end{equation*}
A change of variables gives
\begin{equation*}
	|I_5| \le C \sum_j \frac{r_j^{1+d-1}}{|x - \tilde{x}_j|^{d-1}}.
\end{equation*}
Then by Lemma \ref{lem_Thetat} and a familiar argument, we obtain
\begin{equation*}
	\int_{D_\e} |I_5|^2 \le C \e^{2s \wedge s(d-1)-\delta}.
\end{equation*}

Now it suffices to choose $s \in [1/2,1]$ properly to maximize the exponents for the bounds of $I_k$'s, as well as (\ref{est_Gamma}). To simplify, we note that the bound of $I_5$ is controlled by $I_4$ and the exponent $2s$ in the bound of $I_4$ can be ignored since $2s\ge 1$. As a result, it is sufficient to maximize
\begin{equation}\label{eq_maxbeta}
	\alpha^* = \max_{s\in [1/2,1]} \bigg[ s \wedge 4(s - \frac{1}{2}) \wedge (d-1)(s - \frac{1}{2})
	\wedge \frac{(1-s)(d-1)}{\gamma -1} \wedge \frac{s(d-1)}{1+\gamma}\bigg].
\end{equation}
It is easy to see that $\alpha^*$ is well-defined and $0 <\alpha^* \le 1$, if $\gamma >1$.

If $\gamma = 1$, i.e., $p = d-1$, we should replace (\ref{eq_maxbeta}) by
\begin{equation}\label{eq_max1}
	\alpha^* = \max_{s\in [1/2,1]} \bigg[ s \wedge 4(s - \frac{1}{2}) \wedge (d-1)(s - \frac{1}{2} ) \wedge \frac{s(d-1)}{2}\bigg],
\end{equation}
since the term involving $\gamma -1$ is positive infinity as long as $s \neq 1$. Note that (\ref{eq_max1}) is an increasing function of $s$ and thus the maximum is attained as $s$ approaching $1$. Thus in this case,
\begin{equation}\label{eq_alpha*1}
	\alpha^* = 1 \wedge \frac{d-1}{2} = 1\wedge \frac{p}{2}.
\end{equation}
By Proposition \ref{prop_convex_Lp}, this is exactly the case of strictly convex domains and (\ref{eq_alpha*1}) conincides with (\ref{est_convex_rate}) as expected.

Therefore, we have shown that
\begin{equation*}
	\int_{\Omega} |\tilde{u}_\e - u_0|^2 \le C\e^{\alpha^* - \delta},
\end{equation*}
for arbitrarily small $\delta>0$, where $\alpha^*$ is given by (\ref{eq_maxbeta}) if $p<d-1$ and by (\ref{eq_alpha*1}) if $p = d-1$. This, together with Lemma \ref{lem_uetue} and Proposition \ref{prop_barf_reg} (ii), ends the proof of Theorem \ref{thm_maina}.

\begin{remark}\label{rmk_lowdim}
	Note that the exponent $\alpha^*$ for $\gamma > 1$ in (\ref{eq_maxbeta}) can be computed precisely by solving a linear programming problem. However, for lower dimensional cases ($d\le 5$), we can determine $\alpha^*$ easily. And for higher dimensional cases , it is not hard to find a lower bound for $\alpha^*$. These will be treated separately in the following.
	
	{\bf Case 1: $d=2,3$.} Note that in this situation, the second term and the last term in the brackets of (\ref{eq_maxbeta}) can be ignored since $4\ge d-1$ and $d-1\le 1+\gamma$. For either $d=2$ or $d=3$, the maximum is attained by setting
	\begin{equation*}
		(d-1)(s - \frac{1}{2})
		= \frac{(1-s)(d-1)}{\gamma -1},
	\end{equation*}
	which gives $s = (1+\gamma)/(2\gamma)$. Substituting this back we obtain that
	\begin{equation*}
		\alpha^* = \frac{d-1}{2\gamma} = \frac{p}{2}.
	\end{equation*}
	
	{\bf Case 2: $d = 4,5$.} In this situation, only the second term in the brackets of (\ref{eq_maxbeta}) can be ignored. And we need to consider two subcases. If $d-1\le \gamma+1$, i.e., $\gamma \ge d-2$, then
	\begin{equation*}
		\alpha^* = \max_{s\in [1/2,1]} \bigg[ (d-1)(s - \frac{1}{2})
		\wedge \frac{(1-s)(d-1)}{\gamma -1} \wedge \frac{s(d-1)}{1+\gamma}\bigg] = \frac{p}{2},
	\end{equation*}
	since all the three terms are equal when $s = (1+\gamma)/(2\gamma)$. Now if $1\le \gamma < d-2$,
	\begin{equation*}
		\alpha^* = \max_{s\in [1/2,1]} \bigg[s \wedge (d-1)(s - \frac{1}{2})
		\wedge \frac{(1-s)(d-1)}{\gamma -1}\bigg].
	\end{equation*}
	In this subcase, it is not hard to verify the intersections of the graph for three corresponding linear functions and obtain the maximum point by solving
	\begin{equation*}
		s = \frac{(1-s)(d-1)}{\gamma -1}.
	\end{equation*}
	This gives
	\begin{equation*}
		\alpha^* = s = \frac{d-1}{\gamma + d-2} = \frac{(d-1)p}{d-1+(d-2)p}.
	\end{equation*}
	Combining the two sub-cases together, we have for both $d = 4$ and $5$,
	\begin{equation*}
		\alpha^* = \frac{p}{2} \wedge \frac{(d-1)p}{d-1+(d-2)p}.
	\end{equation*}
	
	{\bf Case 3: $d>5$.} In view of Case 1 and Case 2, it is natural to pick $s = (1+\gamma)/(2\gamma)$, which actually optimizes the last three terms of (\ref{eq_maxbeta}). Then we have
	\begin{equation*}
		\alpha^* \ge \frac{1+\gamma}{2\gamma} \wedge \frac{4}{2\gamma} \wedge \frac{d-1}{2\gamma} = \frac{p+d-1}{2(d-1)} \wedge \frac{2p}{d-1} \wedge \frac{p}{2}.
	\end{equation*}
\end{remark}

\begin{proof}[Proof of Theorem \ref{thm_main}]
	If $\Omega$ is a smooth compact domain of type $k$, then Proposition \ref{prop_typek} claims that $\varkappa(\cdot)^{-1} \in L^{p,\infty}(\partial\Omega,d\sigma)$ with $p = 1/(k-1)$. Then Theorem \ref{thm_main} follows readily from Theorem \ref{thm_maina}.
\end{proof}

\begin{remark}\label{rmk_finite}
	For the domains of finite type, the rate of convergence in Theorem \ref{thm_main} is not nearly close to the result of Theorem \ref{thm_const}, except for $k=2$ and lower dimensions. Actually, if $2\le d\le 5$ and $\partial\Omega$ is of type $k$, by Proposition \ref{prop_typek}, $\varkappa(n(\cdot))^{-1} \in L^{p,\infty}$ with $p = 1/(k-1)$. Then applying Theorem \ref{thm_maina} and Remark \ref{rmk_lowdim}, we have
	\begin{equation*}
		\norm{u_\e - u_0}_{L^2(\Omega)} \le C\e^{\frac{1}{4(k-1)} - \delta}.
	\end{equation*}
	The only case coincides with Theorem \ref{thm_const} is $k =2$.
\end{remark}

We end this paper by stating an application of Theorem \ref{thm_main} concerning the higher order convergence rate for non-oscillating Dirichlet boundary value problem. 
\begin{theorem}\label{thm_high}
	Let $A$ and $\Omega$ be the same as Theorem \ref{thm_main}. Assume that $u_\e$ is the solution of
	\begin{equation}\label{eq_Leg}
		\left\{
		\begin{aligned}
			\cL_\e u_\e(x) &= 0 &\quad & \txt{in } \Omega, \\
			u_\e(x) &= g(x) &\quad & \txt{on } \partial\Omega,
		\end{aligned}
		\right.
	\end{equation}
	where $g$ is smooth. Let $u_0$ be the solution of the homogenized problem of (\ref{eq_Leg}). Then there exist a unique function $v^{bl}$ independent of $\e$ such that
	\begin{equation}\label{eq_L2high}
		\norm{u_\e(x) - u_0(x) - \e \chi(x/\e) \nabla u_0(x) - \e v^{\txt{bl}}(x)}_{L^2(\Omega)} \le C\e^{1+\alpha^*-\delta},
	\end{equation}
	for any $\delta>0$, where $\alpha^*$ is given by (\ref{eq_alpha*}). Furthermore, the function $v^{\txt{bl}}$ is the solution of a non-oscillating Dirichlet problem
	\begin{equation*}
		\left\{
		\begin{aligned}
			\cL_0 v^{\txt{bl}}(x) &= 0 &\quad & \txt{in } \Omega, \\
			v^{\txt{bl}}(x) &= g_*(x) &\quad & \txt{on } \partial\Omega,
		\end{aligned}
		\right.
	\end{equation*}
	which is the homogenized problem of
	\begin{equation}\label{eq_Lebl}
		\left\{
		\begin{aligned}
			\cL_\e u_{1,\e}^{\txt{bl}}(x) &= 0 &\quad & \txt{in } \Omega, \\
			u_{1,\e}^{\txt{bl}}(x) &= -\chi\Big(\frac{x}{\e}\Big) \nabla u_0(x) &\quad & \txt{on } \partial\Omega.
		\end{aligned}
		\right.
	\end{equation}
\end{theorem}

Note that (\ref{eq_Lebl}) is a special case of (\ref{eq_Leue}). The proof of Theorem \ref{thm_high} is the same as \cite[Theorem 5.1]{GerMas12}. From the improved $L^2$ convergence rate (\ref{eq_L2high}), one can also have an improved $H^1$ convergence rate $O(\e^{1+\alpha^*-\delta})$ in any relatively compact subset of $\Omega$. The details will be omitted here; see \cite{GerMas12} for reference.

\bibliographystyle{amsplain}
\bibliography{mybib}
\end{document}